\documentclass[onefignum,onetabnum]{siamonline220329}

\usepackage{lipsum}
\usepackage{amsfonts}
\usepackage{graphicx}
\usepackage{epstopdf}
\usepackage{algorithmic}
\usepackage{mathtools}
\usepackage{mathrsfs}
\usepackage{amssymb}
\usepackage{booktabs}
\usepackage{multirow}
\usepackage{amsmath}
\ifpdf
  \DeclareGraphicsExtensions{.eps,.pdf,.png,.jpg}
\else
  \DeclareGraphicsExtensions{.eps}
\fi

\newcommand{\RR}{{\mathbb R}}

\newcommand{\bS}{{\mathbb S}}

\newcommand{\N}{\mathbb{N}}

\newcommand{\K}{\mathbf{K}}

\newcommand{\bg}{\mathbf{g}}

\newcommand{\bA}{\mathbf{A}}
\def\bx{{\boldsymbol{x}}}

\def\bw{{\boldsymbol{w}}}
\def\bv{{\boldsymbol{v}}}
\def\bu{{\boldsymbol{u}}}
\newcommand{\bX}{\mathbf{X}}

\def\bX{{\boldsymbol{X}}}
\def\bA{{\boldsymbol{A}}}

\def\bu{{\boldsymbol{u}}}

\def\bv{{\boldsymbol{v}}}
\def\bg{{\boldsymbol{\gamma}}}

\newcommand{\sos}{\mbox{\upshape\tiny sos}}

\newcommand{\sps}{\mbox{\upshape\tiny sps}}

\newcommand{\tr}[1]{\mbox{\upshape tr}\left(#1\right)}
\def\ba{{\boldsymbol{\alpha}}}
\def\bb{{\boldsymbol{\beta}}}

\newcommand{\mH}{\mathcal{H}}

\newcommand{\wt}[1]{\widetilde{#1}}

\newcommand{\ve}{\text{\upshape vec}}
\newcommand{\snd}[1]{\vert\N^n_{#1}\vert}

\newcommand{\nul}{\text{\upshape Null}}

\newif\ifcomment
\commentfalse
\commenttrue

\usepackage{enumitem}
\setlist[enumerate]{leftmargin=.5in}
\setlist[itemize]{leftmargin=.5in}

\newsiamremark{remark}{Remark}
\newsiamremark{hypothesis}{Hypothesis}
\crefname{hypothesis}{Hypothesis}{Hypotheses}
\newsiamthm{claim}{Claim}
\newsiamremark{assumption}{Assumption}
\newsiamremark{example}{Example}

\headers{Non-SOS Positivstellens\"atze for semialgebraic sets defined by PMIs}{F. Guo}

\title{Non-SOS Positivstellens\"atze for semi-algebraic sets defined by polynomial matrix inequalities\thanks{Submitted to the editors DATE.
\funding{Feng Guo was supported by the Chinese National Natural Science Foundation under grant 12471478.}}}

\author{Feng Guo\thanks{School of Mathematical Sciences, Dalian University of Technology, Dalian 116024, Liaoning Province, China (\email{fguo@dlut.edu.cn}).}}
\usepackage{amsopn}
\DeclareMathOperator{\diag}{diag}

\ifpdf
\hypersetup{
  pdftitle={An Example Article},
  pdfauthor={D. Doe, P. T. Frank, and J. E. Smith}
}
\fi

\begin{document}

\maketitle

\begin{abstract}
This paper establishes new Positivstellens\"atze for polynomials that are positive on sets 
defined by polynomial matrix inequalities (PMIs). 
We extend the classical Handelman and Krivine–Stengle theorems from the scalar inequality 
setting to the matrix context, deriving explicit certificate forms that do not rely on 
sums-of-squares (SOS). 
Specifically, we show that under certain conditions,
any polynomial positive on a PMI-defined semialgebraic set
admits a representation using Kronecker powers of the defining matrix (or its dilated form)  
with positive semidefinite coefficient matrices.
Under correlative sparsity pattern, we further prove more efficient, 
sparse representations that significantly reduce computational complexity.
By applying these results to polynomial optimization with PMI constraints, 
we construct a hierarchy of semidefinite programming relaxations whose size depends only on 
the dimension of the constraint matrix, and not on the number of variables.
Consequently, our relaxations may remain computationally feasible for problems 
with large number of variables and low-dimensional matrix constraints, 
offering a practical alternative where the traditional SOS-based relaxations become intractable.
\end{abstract}



\section{Introduction}

Positivstellens\"atze are fundamental results in real algebraic geometry, 
offering constructive certificates for the positivity of a polynomial over a semialgebraic set
\cite{realAG,Laurent_sumsof,PPSOSMarshall,ScheidererSurvey}. 
These powerful results have widespread applications in areas like optimization, algebraic geometry, control theory, and more where polynomial positivity is a key concern \cite{SOCAGbook,HKL2020,Lasserre2015,Laurent_sumsof,Nie2023}.

Typically, Positivstellens\"atze characterize a polynomial 
$f(x)\in\RR[\bx]:=\RR[x_1,\ldots,x_n]$ that is positive on a basic semialgebraic set
\begin{equation}\label{eq::set1}
\{\bx\in\RR^n \mid g_j(\bx)\ge 0, j=1,\ldots,s\},
\end{equation}
where $g_1,\ldots,g_s\in\RR[\bx]$.
Many such results achieve this by expressing the polynomial using
sums of squares (SOS) structures. For instance,
under the assumption that the set \eqref{eq::set1} is compact, 
Schm\"udgen \cite{Schmugen1991} showed that any polynomial $f$ positive on the set \eqref{eq::set1}
belongs to the preordering generated by $g_j$'s--that is, it can be written as an SOS-weighted 
combination of cross-products of $g_j$'s. 
Putinar’s Positivstellensatz \cite{Putinar1993} provides an alternative representation that avoids 
the need for cross-products of $g_j$'s under the Archimedean condition (slightly stronger than compactness). 
Building on Putinar’s Positivstellensatz, Lasserre \cite{LasserreGlobal2001,Lasserre09,Lasserre2015} 
introduced the well-known moment-SOS hierarchy of semidefinite programming (SDP) relaxations 
for polynomial optimization. 
Reznick \cite{Reznick1995} proved that after be multiplying by some power of $\sum_{i=1}^n x_i^2$, any 
positive definite form is a sum of even 
powers of linear forms. Putinar and Vasilescu \cite{PV1999} extended Reznick's result to the constrained case where $f$ and $g_j$'s are homogeneous polynomial of even degree.

There also exist several non-SOS structured Positivstellens\"atze.
P\'olya \cite{polya} demonstrated that if $f$ is homogeneous and positive on 
$\RR^n_+\setminus\{\mathbf{0}\}$, then 
multiplying $f$ by some power of $\sum_{i=1}^nx_i$, one obtains a polynomial with nonnegative coefficients. 
Dickinson and Povh \cite{DP2015} generalized P\'olya's Positivstellensatz for homogeneous polynomials 
being positive on the intersection of 
$\RR^n_+\setminus\{\mathbf{0}\}$ and the set \eqref{eq::set1}, where $g_j$'s are assumed to be also homogeneous. 
By replacing SOS polynomials by sums of non-negative circuit polynomials 
and sums of arithmetic-geometric mean exponential functions,
non-SOS Schm\"udgen-type and Putinar-type Positivstellens\"atze have been derived in 
\cite{CS2016, DId2017,RVZ2021}. 
When the set \eqref{eq::set1} is a compact polyhedron with non-empty interior, Handelman's Positivstellensatz
\cite{Handelman1988} states
that any polynomial $f$ positive on the set \eqref{eq::set1} can be expressed as
\begin{equation}\label{eq::handelman}
f(\bx)=\sum_{\ba\in\N^s}\lambda_{\ba}g_1(\bx)^{\alpha_1}\cdots g_s(\bx)^{\alpha_s},
\end{equation}
for some (finitely many) nonnegative scalars $\{\lambda_{\ba}\}$. In general, 
assuming that $0\le g_i(\bx)\le 1$ on the set \eqref{eq::set1} for every $i$, and the family 
$\{1, g_1,\ldots, g_s\}$ generates $\RR[\bx]$, Krivine-Stengle's Positivstellensatz 
\cite{Krivine1964, Stengle1974} (see also \cite{Vas2003})
asserts that a positive polynomial $f$ on the set \eqref{eq::set1} 
admits a representation of the form
\begin{equation}\label{eq::KS}
f(\bx)=\sum_{\ba, \bb\in\N^s}\delta_{\ba,\bb}g_1(\bx)^{\alpha_1}(1-g_1(\bx))^{\beta_1}\cdots 
g_s(\bx)^{\alpha_s}(1-g_s(\bx))^{\beta_s},
\end{equation}
for some (finitely many) nonnegative scalars $\{\delta_{\ba,\bb}\}$.  
As an alternative to the moment-SOS relaxations for polynomial optimization,
Lasserre \cite{LasserreLP2002,LasserreLP2005} proposed the linear programming (LP) relaxations 
by leveraging the representations \eqref{eq::handelman} and \eqref{eq::KS}.

Most Positivstellens\"atze could be generalized to the matrix setting. 
A matrix version of Handelman's Positivstellensatz was proposed in \cite{LD2018}
to characterize a symmetric polynomial matrix
that is positive definite on the set \eqref{eq::set1}. Using the class of SOS polynomial matrices,
Scherer and Hol \cite{SH2006} developed a matrix analogue of Putinar’s Positivstellensatz 
for polynomial matrices positive definite (PD) on a set defined by a polynomial matrix inequality (PMI):
\begin{equation}\label{eq::K}
    \K:=\{\bx\in\RR^n \mid G(\bx)\succeq 0\},
\end{equation}
where $G$ is a $q\times q$ symmetric polynomial matrix.
Building on this, Dinh et al. \cite{DHL2021} generalized the classical Schm\"udgen, 
Putinar-Vasilescu, and Dickinson–Povh Positivstellens\"atze to the polynomial matrix setting.
Given the projections of two PMI-defined semialgebraic sets, Klep and Nie \cite{KN2020} provided
a matrix Positivstellensatz with lifting polynomials to determine whether one is contained in the other. 
These generalizations have a wide range of applications, particularly in areas 
such as optimal control, systems theory \cite{HL2006,HL2012,ichihara2009optimal,pozdyayev2014atomic,vanantwerp2000tutorial}.

\vskip 5pt
\noindent{\itshape Contributions.}
In this paper, we establish non-SOS Positivstellens\"atze for characterizing 
a polynomial $f(\bx)\in\RR[\bx]$ that is positive on the PMI-defined semialgebraic set $\K$. 
Our results extend the classical Positivstellens\"atze of Handelman \eqref{eq::handelman} 
and  Krivine-Stengle \eqref{eq::KS}--which apply to semialgebraic sets defined by finitely many 
scalar polynomial inequalities--to the matrix-valued setting. 

(i) We first consider the case when $G(\bx)$ is linear in $\bx$. Assuming that 
$\K$ is compact and has nonempty interior, we prove that if $f(\bx)$ is positive on $\K$,
then it has the representation
\begin{equation}\label{eq::linear}
f(\bx)=\lambda_0+\sum_{k=1}^m \left\langle\Lambda_k, G(\bx)^{\otimes k}\right\rangle
\end{equation}
for some $m\in\N$, nonnegative number $\lambda_0$ and ${q^k}\times {q^k}$ positive
semidefinite (PSD) matrices $\Lambda_k$, $k=1,\ldots,m$ (Theorem \ref{th::main1}).

(ii) Then we extend the result to the nonlinear case. Assume that $\K$ is compact and 
$I_q-G(\bx)\succeq 0$ on $\K$ ($I_q$ denotes the $q\times q$ identity matrix), and that
$\RR[\bx]$ can be generated by $\{I_q, G(\bx)\}$ (in the sense of Assumption \ref{assump::1} (ii)),
we prove that if $f(\bx)$ is positive on $\K$, then it can be written as
\begin{equation}\label{eq::nonlinear}
\lambda_0+\sum_{k=1}^m \left\langle\Lambda_k, \wt{G}(\bx)^{\otimes k}\right\rangle\ 
\end{equation}
for some $m\in\N$, nonnegative number $\lambda_0$ and ${(2q)^k}\times {(2q)^k}$ positive
semidefinite matrices $\Lambda_k$, $k=1,\ldots,m$ (Theorem \ref{th::main2}). 
Here, $\wt{G}(\bx)$ denotes the block-diagonal matrix $\diag(G(\bx), I_q-G(\bx))$. 

(iii) Assume that $f(\bx)$ and $G(\bx)$ exhibits the correlative sparsity--i.e., 
$G(\bx)$ is block-diagonal and $\bx$ can be grouped into overlapping subsets 
satisfying running intersection property (Assumption \ref{assump::vs}), 
such that each monomial in $f(\bx)$ and each 
block in $G(\bx)$ depend only on one subset--we derive 
sparse version of the representations \eqref{eq::linear} and \eqref{eq::nonlinear}, that requires
much fewer PSD matrices.

(iv) Applying the representations \eqref{eq::linear} and \eqref{eq::nonlinear} to 
polynomial optimization with PMI constraints, we develop a hierarchy of SDP relaxations and 
analysis their asymptotic behavior and exactness. 
As with the LP relaxations for scalar optimization due to Lasserre \cite{LasserreLP2002,LasserreLP2005},
our SDP relaxations can not be exact in general. In principle, the SOS-based SDP hierarchy derived 
from Scherer–Hol's Positivstellensatz \cite{SH2006} may produce tighter bounds. 
Nevertheless, it is worth noting 
that size of our SDP relaxations depends only on the powers of $q$ and not on
the number $n$ of variables. Consequently,
for the problems with very large $n$ and relatively small $q$ 
that are significantly beyond the capability of the SOS-based SDP relaxation,
solving our SDP relaxations may still yield meaningful lower bounds of the optimal value 
in a reasonable time (see Example \ref{ex::3}).

\vskip 7pt
The rest of this paper is organized as follows.  
We recall some notation and preliminaries in section \ref{sec::pre}.
In Section \ref{sec::pos}, we establish the Positivstellens\"atze \eqref{eq::linear}
and \eqref{eq::nonlinear}, along with their sparse version under the correlative 
sparsity. In Section~\ref{sec::apps}, we apply these results to develop SDP relaxations 
for PMI-constrained polynomial optimization. Conclusions are given in Section~\ref{sec::con}.

\section{Notation and preliminaries}\label{sec::pre}
We collect some notation and basic concepts
which will be used in this paper. We denote by $\bx$
the $n$-tuple of variables $(x_1,\ldots,x_n)$.
The symbol $\N$ (resp., $\RR$, $\RR_+$) denotes
the set of nonnegative integers (resp., real numbers, nonnegative real numbers). 
For positive integer $n\in\N$, denote by $[n]$ the set $\{1,\ldots,n\}$.
Denote by $\RR^q$ (resp. $\RR^{l_1\times l_2}$, $\bS^{q}$, $\bS_+^{q}$, $\bS_{++}^{q}$) 
the $q$-dimensional real vector (resp. $l_1\times l_2$ real matrix, $q\times q$ symmetric real matrix,
$q\times q$ PSD matrix, $q\times q$ PD matrix) space. 
Denote by $\RR^n_+$ the nonnegative orthant of $\RR^n$.
For $\bv\in\RR^q$ (resp., $N\in\RR^{l_1\times l_2}$), the symbol 
$\bv^{\intercal}$ (resp., $N^\intercal$) denotes the transpose of $\bv$ (resp., $N$). 
For a matrix $N\in\RR^{q\times q}$, $\tr{N}$ denotes its trace. 
For two matrices $N_1$ and $N_2$, $N_1\otimes N_2$ denotes the 
Kronecker product of $N_1$ and $N_2$.
For two matrices $N_1$ and $N_2$ of the same size, $\langle N_1, N_2\rangle$ denotes the 
inner product $\tr{N_1^{\intercal}N_2}$ of $N_1$ and $N_2$.
For square matrices $M_1,\ldots,M_s$, the notation $\diag(M_1,\ldots,M_s)$ represents 
the block-diagonal matrix formed by placing $M_1,\ldots,M_s$ along the diagonal.
The notation $I_q$ denotes the $q\times q$ identity matrix.
For any $t\in \RR$, $\lceil t\rceil$ (resp., $\lfloor t\rfloor$) denotes the smallest (resp., largest)
integer that is not smaller (resp., larger) than $t$. 
For a vector $\ba=(\alpha_1,\ldots,\alpha_n)\in\N^n$,
let $\vert\ba\vert=\alpha_1+\cdots+\alpha_n$. 
For a set $A$, we use $\vert A\vert$ to denote its cardinality.
For $k\in\N$, let $\N^n_k\coloneqq\{\ba\in\N^n\mid \vert\ba\vert\le k\}$
and $\snd{k}=\binom{n+k}{k}$ be its cardinality.
For variables $\bx \in \RR^n$ and $\ba\in\N^n$, $\bx^{\ba}$ denotes the monomial
$x_1^{\alpha_1}\cdots x_n^{\alpha_n}$.
Let $\RR[\bx]$ (resp. $\bS[\bx]^q$) denote 
the set of real polynomials (resp. $q\times q$ symmetric real polynomial matrices) in $\bx$.
For $h\in\RR[\bx]$, we denote by $\deg(h)$
its total degree in $\bx$.
For a polynomial matrix $T(\bx)=[T_{ij}(\bx)]$, denote $\deg(T)\coloneqq\max_{i,j}\deg(T_{ij})$.
For $k\in\N$, denote by $\RR[\bx]_k$ (resp., $\bS[\bx]^q_k$) the subset of $\RR[\bx]$ (resp., $\bS[\bx]^q$)
of degree up to $k$.

For a polynomial $f(\bx)\in\RR[\bx]$, if there exist polynomials $f_1(\bx),\ldots,f_t(\bx)$ such that $f(\bx)=\sum_{i=1}^tf_i(\bx)^2$,
then we call $f(\bx)$ an SOS polynomial.
A polynomial matrix $\Sigma(\bx)\in\bS[\bx]^{q}$ is said to be an \emph{SOS matrix} if there exists an $l\times q$ polynomial matrix $T(\bx)$ for some
$l\in\N$ such that $\Sigma(\bx)=T(\bx)^{\intercal}T(\bx)$. For $d\in\N$, denote by $[\bx]_d$
the canonical basis of $\RR[\bx]_d$, i.e.,
\begin{equation}\label{eq::ud}
	[\bx]_d\coloneqq[1,\ x_1,\ x_2,\ \cdots,\ x_n,\ x_1^2,\ x_1x_2,\ \cdots,\
	x_n^d]^{\intercal},
\end{equation}
whose cardinality is $\snd{d}=\binom{n+d}{d}$. 
With $d=\deg(T)$, we can write $T(\bx)$ as
\[
	T(\bx)=Q([\bx]_d\otimes I_q) \text{ with } 
 Q=[Q_1,\ldots,Q_{\snd{d}}],\quad Q_i\in\RR^{l\times q},
\]
where $Q$ is the vector of coefficient matrices of $T(\bx)$ with respect to
$[\bx]_d$. Hence, $\Sigma(\bx)$ is an SOS matrix with respect to $[\bx]_d$ if there
exists some $Q\in\RR^{l\times q\snd{d}}$ satisfying 
\[
	\Sigma(\bx)=T(\bx)^{\intercal}T(\bx)=([\bx]_d\otimes I_q)^{\intercal}(Q^{\intercal}Q)([\bx]_d\otimes I_q). 
\]
We thus have the following results.
\begin{proposition}{\upshape \cite[Lemma 1]{SH2006}}\label{prop::SOSrep}
A polynomial matrix $\Sigma(\bx)\in\bS[\bx]^{q}$ is an SOS matrix with
respect to the monomial basis $[\bx]_d$ if and only if there exists $Z\in\mathbb{S}_+^{q\snd{d}}$ such that  
\[\Sigma(\bx)=([\bx]_d\otimes I_q)^{\intercal} Z ([\bx]_d\otimes I_q).\]
\end{proposition}


We recall Scherer-Hol's Positivstellensatz \cite{SH2006} for PMI-defined semialgebraic sets.

\begin{assumption}\label{assump::archi}
{\rm
For the defining matrix $G(\bx)$ of $\K$ in \eqref{eq::K}, there exists $r\in\RR$ and an SOS polynomial matrix 
$\Sigma(\bx)\in\bS[\bx]^{q}$ such that $r^2 - \sum_{i=1}^n x_i^2 - \langle \Sigma(\bx), G(\bx)\rangle$ 
is an SOS.}
\end{assumption}

\begin{theorem}{\upshape \cite[Corollary 1]{SH2006}}\label{th::psatz}
Let Assumption \ref{assump::archi} hold and $f(\bx)\in\RR[\bx]$ be positive definite on $\K$.
Then there exist SOS polynomial $\sigma(\bx)\in\RR[\bx]$ and SOS polynomial matrix
$\Sigma(\bx)\in\bS[\bx]^q$ such that 
\[
f(\bx)=\sigma(\bx)+\left\langle \Sigma(\bx), G(\bx)\right\rangle.
\]
\end{theorem}

\section{Non-SOS Positivstellens\"atze for PMI constraints}\label{sec::pos}
In this section, we characterize polynomials positive on the PMI-defined set $\K$ 
by constructing a novel polynomial cone using Kronecker powers of the defining matrix $G(\bx)$
(or its dilated form)  with PSD coefficient matrices. 
Our results extend the classical Positivstellens\"atze of Handelman \eqref{eq::handelman}
and  Krivine-Stengle \eqref{eq::KS} from scalar polynomial inequalities to 
the matrix-valued setting.
Our proof strategy, inspired by \cite{Vas2003}, involves adapting the scalar framework to 
the matrix setting and integrating our novel cone construction.

\subsection{The linear case}
In this part, we assume that the defining matrix $G(\bx)$ of $\K$ is linear in $\bx$, 
i.e., there exist  matrices $A_0, \ldots, A_n\in\bS^q$, such that 
\begin{equation}
    \K=\left\{\bx\in\RR^n\ \Big|\ G(\bx)=A_0+\sum_{i=1}^n A_ix_i\succeq 0\right\}.
\end{equation}

Recall the polar of the convex set $\K$ defined as follows 
\[
\K^{\circ}=\{\bu\in\RR^n \mid \bu^\intercal \bx\le 1,\ \forall \bx\in\K\}.
\]
\begin{proposition}\label{prop::polar}
    Assume that $\mathbf{0}$ is an interior of $\K$, then 
    \[
    \K^{\circ}=\{-(\langle U, A_1\rangle, \ldots, \langle U, A_n\rangle) \mid 
    U\in\bS^q_+,\ \langle U, A_0\rangle\le 1\}.
    \]
\end{proposition}
\begin{proof}
    Since $\mathbf{0}$ is an interior of $\K$, we have $A_0\succeq 0$ and the affine hull of $\K$ is 
    $\RR^n$. Then, the conclusion follows from \cite[Theorem 2]{RG1995}.
\end{proof}

For any $k\in\N$ and $U\in\bS^q$, 
denote by $U^{\otimes k}$ the $k$-times Kronecker product of $U$, i.e., 
\[
U^{\otimes k}=\underbrace{U\otimes U \otimes \cdots \otimes U}_{k\ \text{times}}.
\]
\begin{assumption}\label{assump::0}
    The set $\K$ is compact and has nonempty interior.
\end{assumption}
\begin{proposition}\label{prop::main}
Suppose that Assumption \ref{assump::0} holds and $\bu$ is an interior of $\K$. 
    Let $\bb(\bx)=\beta_0+\sum_{i=1}^n\beta_i x_i$, $\beta_i\in\RR$. 
    \begin{enumerate}
       \item[\upshape (i)] If $\bb(\bx)\ge 0$ on $\K$, then there exist $U\in\bS^q_+$ and $0\le \gamma\le \bb(\bu)$ such that  
       $\bb(\bx)=\langle U, G(\bx)\rangle+\gamma$;
       \item[\upshape (ii)] There exists $U\in\bS^q_{++}$ such that $1=\langle U, G(\bx)\rangle$;
       \item[\upshape (iii)] For any $k\in\N$, there exists $U\in\bS^{q^k}_{++}$ such that 
       $1=\left\langle U, G(\bx)^{\otimes k}\right\rangle$;
       \item[\upshape (iv)] For each $i\in[n]$, there exists $U_i\in\bS^q$ such that $x_i=\langle U_i, G(\bx)\rangle$.
    \end{enumerate}
\end{proposition}
\begin{proof}
    (i) Without loss of generality, we may assume that $\bu=\mathbf{0}$.
    Since $\bb(\bx)\ge 0$ on $\K$ and $\mathbf{0}\in\K$, we have $\beta_0\ge 0$. If $\beta_0=0$, then $\bb(\bx)$ would be identically 
    zero since $\mathbf{0}$ is an interior of $\K$, and there would be 
    nothing to do. We may thus assume that $\beta_0>0$. As for all $\bx\in\K$,
    \[
    \bb(\bx)=\beta_0\left(1+\sum_{i=1}^n\frac{\beta_i}{\beta_0}x_i\right)\ge 0,
    \]
    we have $-(\beta_1/\beta_0,\ldots,\beta_n/\beta_0)\in \K^{\circ}$.
    By Proposition \ref{prop::polar}, there exists $U'\in\bS^q_+$ such that 
       $\beta_i/\beta_0=\langle U', A_i\rangle$, $i\in[n]$, and 
       $\langle U', A_0\rangle\le 1$. Thus,
       \[
       \bb(\bx)=\beta_0\left(1+\sum_{i=1}^n \langle U', A_i\rangle x_i\right)
       =\langle \beta_0 U', G(\bx)\rangle + \beta_0(1 - \langle U', A_0\rangle).
       \]
The conclusion follows by letting $U=\beta_0 U'$ and $\gamma=\beta_0(1 - \langle U', A_0\rangle)$.

    (ii) Since $\K$ is compact, there exists $\varepsilon>0$ such that 
    $1-\varepsilon\langle I_q, G(\bx)\rangle\ge 0$ on $\K$. By (i), there
    $U'\in\bS^q_+$ and $0\le \gamma\le 1-\varepsilon \langle I_q, G(\bu)\rangle$ such that  
    \[
    1-\varepsilon\langle I_q, G(\bx)\rangle=\langle U', G(\bx)\rangle+\gamma.
    \]
    Then, it holds $1-\gamma=\langle \varepsilon I_q+U', G(\bx)\rangle$. 
    Since $\bu$ is an interior of $\K$, we must have $\langle I_q, G(\bu)\rangle>0$ and 
    thus $\gamma<1$. Then, the conclusion follows by letting $U=(\varepsilon I_q+U')/(1-\gamma)$.

    (iii) By (ii), there exists $U'\in\bS^q_{++}$ such that $1=\langle U', G(\bx)\rangle$. 
    For any $k\in\N$, let $U=U'^{\otimes k}$, then $U\succ 0$ and 
    \[
    \begin{aligned}
    \left\langle U, G(\bx)^{\otimes k}\right\rangle
    &=\tr{U'^{\otimes k}G(\bx)^{\otimes k}}\\
    &=\tr{\left(U'^{\otimes k-1}G(\bx)^{\otimes k-1}\right) \otimes (U'G(\bx))}\\
    &=\tr{(U'^{\otimes k-1}G(\bx)^{\otimes k-1}}\tr{U'G(\bx)}\\
    &=\left\langle U'^{\otimes k-1}, G(\bx)^{\otimes k-1}\right\rangle
    \left\langle U', G(\bx)\right\rangle\\
    &=\left\langle U'^{\otimes k-1}, G(\bx)^{\otimes k-1}\right\rangle
    \end{aligned}
    \]
Then, the conclusion follows by induction on $k$.

    (iv) Fix an $i\in[n]$. Since $\K$ is compact, there exists positive 
    $M_i\in\RR$ such that $x_i\ge -M_i$ for all $\bx\in\K$. Then, by (i),
    there exists $U'\in\bS^q_+$ and $0\le \gamma\le M_i$ such that  
    $x_i+M_i=\langle U', G(\bx)\rangle+\gamma$. By (ii), there exists $U''\in\bS^q$ such that  
    $\gamma-M_i=\langle U'', G(\bx)\rangle$. Thus, 
    \[
    x_i=(x_i+M_i) - M_i=\langle U'+U'', G(\bx)\rangle.
    \]
    The conclusion follows by letting $U=U'+U''$.
\end{proof}

Let 
\[
\mH(G):=\left\{\lambda_0+\sum_{k=1}^m \left\langle\Lambda_k, G(\bx)^{\otimes k}\right\rangle\ \Big|\ m\in\N,\ \lambda_0\in\RR_+,\ \Lambda_k\in\bS^{q^k}_+,\ k\in[m]\right\}\subset\RR[\bx].
\]
Clearly, if $f\in\mH(G)$, then $f(\bx)\ge 0$ on $\K$. 
\begin{proposition}\label{prop::closed}
The cone $\mH(G)$ is closed under addition and multiplication.
\end{proposition}
\begin{proof}
It is obvious that $f+g\in\mH(G)$ if $f, g\in\mH(G)$. To see $f\cdot g\in \mH(G)$, 
    notice that for any $k, \ell\in\N$, it holds
    \[
    \begin{aligned}
    \left\langle\Lambda_k, G(\bx)^{\otimes k}\right\rangle
    \left\langle\Lambda_{\ell}, G(\bx)^{\otimes \ell}\right\rangle
    &=\tr{\Lambda_k G(\bx)^{\otimes k}}\tr{\Lambda_{\ell}G(\bx)^{\otimes \ell}}\\
    &=\tr{\left(\Lambda_k G(\bx)^{\otimes k}\right)\otimes \left(\Lambda_{\ell}G(\bx)^{\otimes \ell}\right)}\\
    &=\tr{\left(\Lambda_k \otimes \Lambda_{\ell} \right)\left(G(\bx)^{\otimes k}\otimes G(\bx)^{\otimes \ell}\right)}\\
    &=\left\langle\Lambda_k\otimes \Lambda_{\ell}, G(\bx)^{\otimes k+\ell}\right\rangle,
    \end{aligned}
    \]
    and $\Lambda_k\otimes \Lambda_{\ell}\in \bS^{q^{k+\ell}}_+$ for any $\Lambda_k\in\bS^{q^k}_+$
    and $\Lambda_{\ell}\in\bS^{q^{\ell}}_+$.
\end{proof}

Moreover, we have 
\begin{proposition}\label{prop::pc}
    If Assumption \ref{assump::0} holds, then for any 
    $h\in\RR[\bx]$, there exist $h_1, h_2\in\mH(G)$ such that 
    $h=h_1-h_2$; that is $\RR[\bx]=\mH(G)-\mH(G)$.
\end{proposition}
\begin{proof}

    Now suppose that Assumption \ref{assump::0} holds. To prove $\RR[\bx]=\mH(G)-\mH(G)$,
    without loss of generality,  it suffices to prove that $\bx^\ba\in \mH(G)-\mH(G)$ 
    for any $\bx^\ba=x_1^{\alpha_1}\cdots x_n^{\alpha_n}$, $\ba\in\N^n$.
   By Proposition \ref{prop::main} (iv), there exists $U_i\in\bS^q$ such
   that $x_i=\langle U_i, G(\bx)\rangle$. Then, for each $i\in[n]$ and $\alpha_i\in\N$,
   \[
   x_i^{\alpha_i}=\left\langle U_i^{\otimes{\alpha_i}}, G(\bx)^{\otimes{\alpha_i}}\right\rangle,
   \]
   and 
   \[
   \bx^{\ba}=x_1^{\alpha_1}\cdots x_n^{\alpha_n}
   =\left\langle U_1^{\otimes{\alpha_1}}\otimes \cdots\otimes U_n^{\otimes{\alpha_n}}, 
   G(\bx)^{\otimes{|\ba|}}\right\rangle.
   \]
   Decompose $U_1^{\otimes{\alpha_1}}\otimes \cdots\otimes U_n^{\otimes{\alpha_n}}=U^+-V^+$,
   where $U^+, V^+\in\bS_+^{q^{|\ba|}}$. Then, 
   \[
   \bx^{\ba}
   =\left\langle U^+,  G(\bx)^{\otimes{|\ba|}}\right\rangle - 
   \left\langle V^+,  G(\bx)^{\otimes{|\ba|}}\right\rangle\in\mH(G)-\mH(G).
   \]
   The conclusion follows.
\end{proof}



\begin{definition}
    Let $V$ be a vector space and let $A\subset V$ be a convex set. We 
    say that $\bv\in A$ lies in the algebraic interior of $A$ if for any 
    straight line $\ell$ passing through $\bv$, the point $\bv$ lies in 
    the interior of the intersection $A\cap\ell$. 
\end{definition}

\begin{theorem}\cite[III, Corollary 1.7]{ACC}\label{th::sep}
    Let $V$ be a vector space and let $A, B\subset V$ be non-empty convex sets such that $A\cap B=\emptyset$. Suppose that $A$ has 
    a non-empty algebraic interior. Then, $A$ and $B$ can be separated 
    by an affine hyperplane.
\end{theorem}

\begin{proposition}\label{prop::unperforated}
    Suppose that Assumption \ref{assump::0} holds.
    Then, $1$ lies in the algebraic interior of the convex set 
    $\mH(G)$.
\end{proposition}
\begin{proof}
By definition, we need to prove that for any $f\in\RR[\bx]$, there 
exists $\varepsilon_f>0$ such that $1+\lambda f\in\mH(G)$ for all 
$|\lambda|<\varepsilon_f$. Without loss of generality, 
it suffices to assume that $f$ is a monomial.

Fixing an $f=\bx^\ba=x_1^{\alpha_1}\cdots x_n^{\alpha_n}$, $\ba\in\N^n$,
   we prove that there exists positive integer $a$ such that 
   $a\pm \bx^{\ba}\in\mH(G)$. 
   As proved in Proposition \ref{prop::pc},
there exists $U_i\in\bS^q$, $i\in[n]$, such that 
   \[
   \bx^{\ba}=\left\langle U_1^{\otimes{\alpha_1}}\otimes \cdots\otimes U_n^{\otimes{\alpha_n}}, 
   G(\bx)^{\otimes{|\ba|}}\right\rangle.
   \]
   By Proposition \ref{prop::main} (iii), there exists $U\in\bS^{q^{|\ba|}}_{++}$ such that 
   $1=\left\langle U, G(\bx)^{\otimes |\ba|}\right\rangle$.
   As $U$ is positive definite, there exists a positive integer $a$ such that
   $a U\pm U_1^{\otimes{\alpha_1}}\otimes \cdots\otimes U_n^{\otimes{\alpha_n}}\succeq 0$. 
   Hence,
   \[
   a\cdot 1\pm \bx^{\ba}=\left\langle aU\pm U_1^{\otimes{\alpha_1}}\otimes \cdots\otimes U_n^{\otimes{\alpha_n}}, 
   G(\bx)^{\otimes{|\ba|}}\right\rangle\in\mH(G).
   \]
   Then, $1+\lambda \bx^{\ba}\in\mH(G)$ for all 
$|\lambda|<1/a$. By letting $\varepsilon_f=1/a$, the proof is complete.
\end{proof}

Denote by $\pi(\mH(G))$ the set of all linear functionals $L : \RR[\bx]\to \RR$
such that $L(1)=1$ and $L(f)\ge 0$ for all $f\in\mH(G)$.  
Clearly, $\pi(\mH(G))$ is a convex set. 
We can identify each $L\in\pi(\mH(G))$ with the point $(L(\bx^{\ba}))_{\ba\in\N^n}\in\RR^{\N^n}$. 
Hence, we regard 
$\pi(\mH(G))$ as a subset of $\RR^{\N^n}$, which we endow with the 
product topology, i.e., the topology of pointwise convergence. Then,
we have
\begin{lemma}
    If $1$ is an algebraic interior point of $\mH(G)$, 
    then $\pi(\mH(G))$ is compact. 
\end{lemma}
\begin{proof}
    Suppose that $\{L_k\}_{k\in\N}\subset\pi(\mH(G))$ and 
    $\lim_{k\to\infty} L_k=L_{\infty}$. Then, by the pointwise convergence, 
    we have $L_{\infty}(1)=1$ and $L_{\infty}(f)\ge 0$ for all $f\in\mH(G)$. 
    Then, $L_{\infty}\in\pi(\mH(G))$ and thus $\pi(\mH(G))$ is closed.
    Since $1$ is an algebraic interior point of $\mH(G)$, for any $\ba\in\N^n$, 
    there exists a positive integer $N_{\ba}$ such that $N_{\ba}\pm\bx^{\ba}\in\mH(G)$.
    Hence, $|L(\bx^{\ba})|\le N_{\ba}$ for any  $L\in\pi(\mH(G))$ and $\ba\in\N^n$.  
    Thus, $\pi(\mH(G))$ is pointwise bounded and hence compact
    by Tychonoff's Theorem. 
\end{proof}

\begin{lemma}\label{lem::L}
   Suppose that $\RR[\bx]=\mH(G)-\mH(G)$ and $1$ is an algebraic interior point of $\mH(G)$.
   If a linear functional $L : \RR[\bx]\to \RR$ satisfies that $L(1)=0$ and $L(h)\ge 0$ 
   for all $h\in\mH(G)$, then $L=0$.
\end{lemma}
\begin{proof}
    As $\RR[\bx]=\mH(G)-\mH(G)$,
    we only need to prove $L=0$ on $\mH(G)$. Fix an arbitrary $f\in\mH(G)$.
    Then $L(f)\ge 0$.
    Since $1$ is an algebraic interior point of $\mH(G)$, 
    there exists a positive integer $N_f$ such that $N_f-f\in\mH(G)$. Hence,
   \[
   0\le L(N_f-f)=N_f L(1)-L(f)=-L(f),
   \]
   which shows that $L(f)=0$.
\end{proof}

\begin{proposition}\label{prop::extre}
Suppose that $\RR[\bx]=\mH(G)-\mH(G)$ and $1$ is an algebraic interior point of $\mH(G)$. 
Then,  for any extreme point $L$ of $\pi(\mH(G))$, there exists 
     $\bu\in\K$ such that $L(f)=f(\bu)$ for all $f\in\RR[\bx]$.
\end{proposition}
\begin{proof}
   We first prove that $L$ is multiplicative on $\RR[\bx]$. That is, $L(fg)=L(f)L(g)$
   for any $f, g\in\RR[\bx]$. As $\RR[\bx]=\mH(G)-\mH(G)$,
   it suffices to assume that $f, g\in\mH(G)$. 
   Fixing arbitrary $f, g\in\mH(G)$, let us prove $L(fg)=L(f)L(g)$. 
   Since $1$ is an algebraic interior point of $\mH(G)$, 
   there exists a positive integer $N_f$ such that $N_f-f\in\mH(G)$. 
   Thus, $0\le L(f)\le L(N_f)=N_f$.

   If $0<L(f)<N_f$, define two linear functionals $L_1$ and $L_2$ on $\RR[\bx]$ as follows
   \[
   L_1(h)=\frac{L(fh)}{L(f)},\quad 
   L_2(h)=\frac{L((N_f-f)h)}{N_f-L(f)}, \ \ \forall h\in\RR[\bx]. 
   \]
   As $N_f-f\in\mH(G)$, we have $L_1, L_2\in\pi(\mH(G))$. Observe that
   \[
   L=\frac{L(f)}{N_f}L_1 + \frac{N_f-L(f)}{N_f}L_2,\quad 
  \frac{L(f)}{N_f}+\frac{N_f-L(f)}{N_f}=1. 
   \]
   As $L$ is an extreme point of $\pi(\mH(G))$, we must have 
   $L=L_1$. Therefore, we have 
   \[
   L(g)=L_1(g)=\frac{L(fg)}{L(f)},
   \]
   which implies $L(fg)=L(f)L(g)$.

   If $L(f)=0$, let $\wt{L}$ be the linear functional on $\RR[\bx]$ such that $\wt{L}(h)=L(fh)$ for all $h\in\RR[\bx]$.
   Then, $\wt{L}(1)=0$ and $\wt{L}(h)\ge 0$ for all $h\in\mH(G)$.
   By Lemma \ref{lem::L}, we have $\wt{L}=0$ and hence 
   \[
   L(fg)=\wt{L}(g)=0=L(f)L(g).
   \]

   If $L(f)=N_f$, let $\wt{L}$ be the linear functional on $\RR[\bx]$ such that $\wt{L}(h)=L((N_f-f)h)$ for all $h\in\RR[\bx]$. Then, repeating the above arguments shows that
   $\wt{L}=0$ and hence,
   \[
   N_fL(g)-L(fg)=L((N_f-f)g)=\wt{L}(g)=0=(N_f-L(f))L(g)=N_fL(g)-L(f)L(g),
   \]
   which implies $L(fg)=L(f)L(g)$.

Letting $\bu=(L(x_1), \ldots, L(x_n))\in\RR^n$, as we have shown that 
$L$ is multiplicative on $\RR[\bx]$,
it holds $L(f)=f(\bu)$ for all $f\in\RR[\bx]$. It remains to show that $\bu\in\K$.
    Fix an arbitrary $\bv\in\RR^q$. As 
    $\bv^\intercal G(\bx)\bv=\langle \bv\bv^\intercal, G(\bx)\rangle\in\mH(G)$, 
    we have 
    \[
    \bv^{\intercal}G(\bu)\bv=\langle \bv\bv^\intercal, G(\bu)\rangle 
    =L(\langle \bv\bv^\intercal, G(\bx)\rangle)\ge 0.
    \]
    Thus, $G(\bu)\succeq 0$ and $\bu\in\K$.
\end{proof}

Now we are ready to prove the following result,
which extends the Handelman's Positivstellensatz \cite{Handelman1988} from the scalar inequality 
setting to the matrix context.
\begin{theorem}\label{th::main1}
    Suppose that Assumption \ref{assump::0} holds. If a polynomial 
    $f(\bx)\in\RR[\bx]$ is positive on $\K$, then $f\in\mH(G)$.
\end{theorem}
\begin{proof}
As Assumption \ref{assump::0} holds, by Propositions \ref{prop::pc} and \ref{prop::unperforated}, 
we have $\RR[\bx]=\mH(G)-\mH(G)$ and $1$ is an algebraic interior point of $\mH(G)$. 

Suppose to the contrary that $f\not\in\mH(G)$. Then, by Proposition \ref{prop::unperforated} and Theorem \ref{th::sep},
there exists a nonzero linear functional $L: \RR[\bx]\to \RR$ 
such that $L(f)\le \inf_{h\in\mH(G)}L(h)$. As $\mH(G)$ is a cone,
we must have $\inf_{h\in\mH(G)}L(h)=0$, which implies that 
$L(f)\le 0$ and $L(h)\ge 0$ for all $h\in\mH(G)$. By Lemma \ref{lem::L}, we have $L(1)>0$ and hence $L/L(1)\in\pi(\mH(G))$.
By applying the Krein-Milman theorem to the convex compact set $\pi(\mH(G))$,
we can conclude that $L/L(1)$ lies in the closed convex hull of
extreme points of $\pi(\mH(G))$. Hence, we have 
$L/L(1)=\lim_{k\to\infty} L_k$ (pointwise convergence), 
where each $L_k$ belongs to the convex hull
of extreme points of $\pi(\mH(G))$. By Proposition \ref{prop::extre}, for each 
$k\in\N$, there exist points $\bu_1,\ldots,\bu_{\ell_k}\in\K$ for some $\ell_k\in\N$ and positive real numbers $\lambda_1,\ldots,\lambda_{\ell_k}\in\RR$ with $\sum_{i=1}^{\ell_k}\lambda_i=1$ such 
that
\[
L_k(h)=\sum_{i=1}^{\ell_k}\lambda_i h(\bu_i),\ \ \forall h\in\RR[\bx].
\]
As $\K$ is compact, there is some $\varepsilon>0$ such that $f(\bx)\ge \varepsilon$ on $\K$. Therefore, 
$L_k(f)=\sum_{i=1}^{\ell_k}\lambda_i f(\bu_k)\ge \varepsilon$ for all
$k\in\N$. By the pointwise convergence, 
\[
L(f)=L(1)\lim_{k\to\infty}L_k(f)\ge L(1)\varepsilon>0,
\]
a contradiction. 
\end{proof}
\begin{remark}\label{rmk::core}
    From the proof, we observe that the key to establishing Theorem \ref{th::main1} 
    lies in the properties: (i) $\mH(G)$ is closed under addition and multiplication;
    (ii) $\RR[\bx]=\mH(G)-\mH(G)$; and (iii) $1$ is an algebraic interior point of $\mH(G)$.
\end{remark}

\subsection{The nonlinear case}

In this part, let $G(\bx)\in\bS[\bx]^q$ be a nonlinear polynomial matrix.
If the set $\K$ defined by the inequality $G(\bx)\succeq 0$ is compact, 
by rescaling if necessary, we may assume that $I_q-G(\bx)\succeq 0$ on $\K$. 
We now introduce the dilated form of $G(\bx)$
\[\wt{G}(\bx)=\diag(G(\bx), I_q-G(\bx))\in \bS[\bx]^{2q}.
\]
Clearly, the inequality $\wt{G}(\bx)\succeq 0$ defines the same set $\K$.

Let 
\[
\mH(G):=\left\{\lambda_0+\sum_{k=1}^m \left\langle\Lambda_k, \wt{G}(\bx)^{\otimes k}\right\rangle\ \Big|\ m\in\N,\ \lambda_0\in\RR_+,\ \Lambda_k\in\bS^{(2q)^k}_+,\ k\in[m]\right\}\subset\RR[\bx].
\]
As with Proposition \ref{prop::closed}, it is readily seen that the cone 
$\mH(G)$ is closed under both addition and multiplication.

\begin{assumption}\label{assump::1}
{\upshape (i)} $\K$ is compact and $I_q-G(\bx)\succeq 0$ on $\K$;
{\upshape (ii)} $\RR[\bx]=\mH(G)-\mH(G)$.
\end{assumption}

Under Assumption \ref{assump::1}, it is clear that $f(\bx)\ge 0$ on $\K$ if $f\in\mH(G)$.
\begin{remark}
Assumption \ref{assump::1} (ii) holds if and only if for any $f(\bx)\in\RR[\bx]$, 
there exist $m\in\N$, $\gamma\in\RR$, and symmetric matrices 
$\Gamma_k\in\bS^{(2q)^k}$, $k\in[m]$, such that 
\[
f(\bx)=\gamma+\sum_{k=1}^m\left\langle \Gamma_k, \wt{G}(\bx)^{\otimes k}\right\rangle.
\]
In this case, we say that $\RR[\bx]$ can be generated by $\{I_q, G(\bx)\}$. 
    If Assumption \ref{assump::1} (ii) does not hold, it can be restored by augmenting $G(\bx)\succeq 0$  
    by some redundant inequalities. In fact, 
    as $\K$ is compact, for any $i\in[n]$, there exists $b_i\in\RR$ such that $x_i-b_i\ge 0$
    on $\K$. Then, it follows directly that Assumption \ref{assump::1} is satisfied when $G(\bx)$ 
    is replaced by $\diag(G(\bx), x_1-b_1,\ldots, x_n-b_n)$. 
\end{remark}


\begin{proposition}\label{prop::unperforated2}
    Suppose that Assumption \ref{assump::1} holds.
    Then, $1$ lies in the algebraic interior of the convex set 
    $\mH(G)$.
\end{proposition}
\begin{proof}
    To prove that $1$ is an algebraic interior point of $\mH(G)$, 
    it suffices to prove that for any monomial $\bx^\ba=x_1^{\alpha_1}\cdots x_n^{\alpha_n}$, $\ba\in\N^n$,
   there exists positive $\lambda_{\ba}\in\RR$ such that  $1\pm t\bx^{\ba}\in\mH(G)$ for all
   $0\le t\le \lambda_{\ba}$. 

   As Assumption \ref{assump::1} (ii) holds, there exist $m\in\N$, $\gamma\in\RR$, and symmetric matrices 
   $\Gamma_k\in\bS^{(2q)^k}$, $k\in[m]$, such that 
\[
\bx^{\ba}=\gamma+\sum_{k=1}^m\left\langle \Gamma_k, \wt{G}(\bx)^{\otimes k}\right\rangle.
\]
Write $\gamma=\gamma^+-\gamma^-$ and $\Gamma_k=\Gamma_k^+-\Gamma_k^-$, $k\in[m]$, where 
$\gamma^{\pm}> 0$, $\Gamma_k^{\pm}\succeq 0$, $k\in[m]$.

Letting $t\in\RR$ be positive such that 
\begin{equation}\label{eq::t}
0<t\le \left(\gamma^++\sum_{k=1}^m\tr{\Gamma_k^+}\right)^{-1}.
\end{equation}
We first show that $1-t\bx^{\ba}\in\mH(G)$. 

We have 
\[
\begin{aligned}
1-t\bx^{\ba}&=1-\left(t\gamma^++\sum_{k=1}^m\left\langle t\Gamma_k^+, \wt{G}(\bx)^{\otimes k}\right\rangle\right)
+\left(t\gamma^-+\sum_{k=1}^m\left\langle t\Gamma_k^-, \wt{G}(\bx)^{\otimes k}\right\rangle\right)\\
&=\left(1-t\gamma^++\sum_{k=1}^m\left\langle -t\Gamma_k^+, \wt{G}(\bx)^{\otimes k}\right\rangle\right)
+\left(t\gamma^-+\sum_{k=1}^m\left\langle t\Gamma_k^-, \wt{G}(\bx)^{\otimes k}\right\rangle\right).\\
\end{aligned}
\]
Observe that 
\[
\begin{aligned}
1-t\gamma^++\sum_{k=1}^m\left\langle -t\Gamma_k^+, \wt{G}(\bx)^{\otimes k}\right\rangle
&=1-t\gamma^++\sum_{k=1}^m\left\langle -t\Gamma_k^+, \wt{G}(\bx)^{\otimes k}-I_{(2q)^k}+I_{(2q)^k}\right\rangle\\
&=1-t\left(\gamma^++\sum_{k=1}^m\tr{\Gamma_k^+}\right)+t\sum_{k=1}^m\left\langle \Gamma_k^+, I_{(2q)^k}-\wt{G}(\bx)^{\otimes k}\right\rangle.
\end{aligned}
\]
By \eqref{eq::t}, to show $1-t\bx^{\ba}\in\mH(G)$, it remains to prove that for each $k\in[m]$, 
\[
\left\langle \Gamma_k^+, I_{(2q)^k}-\wt{G}(\bx)^{\otimes k}\right\rangle\in\mH(G).
\]

Consider the case when $k=1$. Observe that $I_{2q}-\wt{G}(\bx)=\diag(I_{q}-G(\bx), G(\bx))$.
Equally partition $\Gamma_1^+$ as a $2\times 2$ block matrix $[(\Gamma_1^k)_{ij}]_{i,j\in[2]}$, where 
$(\Gamma_1^+)_{ij}$ denotes its $(i,j)$-th block. Let $\wt{\Gamma}_1^+=\diag((\Gamma_1^+)_{22}, (\Gamma_1^+)_{11})$. 
Then, we have $\wt{\Gamma}_1^+\succeq 0$ and 
\[
\left\langle \Gamma_1^+, I_{2q}-\wt{G}(\bx)\right\rangle=
\left\langle \wt{\Gamma}_1^+, \wt{G}(\bx)\right\rangle\in\mH(G).
\]

Now fix a $k\in[m]$ with $k\ge 2$. We define matrices $\Lambda_\ell\in\bS^{(2q)^\ell}$,
$\ell\in[k]$, in the following way: (i) Let $\Lambda_k=\Gamma_k^+$;
(ii) For $\ell=k,\ldots,2$, 
equally partition $\Lambda_\ell$ as a $2q\times 2q$ block matrix 
$[(\Lambda_\ell)_{ij}]_{i,j\in[2q]}$, where 
$(\Lambda_\ell)_{ij}$ denotes its $(i,j)$-th block; (iii) Let
$\Lambda_{\ell-1}=\sum_{i=1}^{2q}(\Lambda_\ell)_{ii}\in\bS^{(2q)^{\ell-1}}$. 
One can check that the following properties hold:
\begin{enumerate}
    \item[\upshape (i)] $\Lambda_{\ell}\succeq 0$ for all $\ell\in[k]$;
    \item[\upshape (ii)] $\left\langle \Lambda_1, I_{2q}\right\rangle=\left\langle \Lambda_k, I_{(2q)^{k}}\right\rangle=\left\langle \Gamma_k^+, I_{(2q)^k}\right\rangle$;
    \item[\upshape (iii)] For each $\ell=2,\ldots,k$, 
    \[
\left\langle \Lambda_{\ell}, I_{2q}\otimes \wt{G}(\bx)^{\otimes \ell-1}\right\rangle
=\left\langle \Lambda_{\ell-1}, \wt{G}(\bx)^{\otimes \ell-1}\right\rangle.
\]
\end{enumerate}
Consequently, 
\begin{equation}\label{eq::LG}
\begin{aligned}
&\left\langle \Lambda_1, I_{2q}-\wt{G}(\bx)\right\rangle+
\sum_{\ell=2}^{k}\left\langle \Lambda_{\ell}, \left(I_{2q}-\wt{G}(\bx)\right)\otimes \wt{G}(\bx)^{\otimes \ell-1}\right\rangle\\
=&\left\langle \Lambda_1, I_{2q}-\wt{G}(\bx)\right\rangle+
\sum_{\ell=2}^{k}\left(\left\langle \Lambda_{\ell}, I_{2q}\otimes \wt{G}(\bx)^{\otimes \ell-1}\right\rangle-\left\langle \Lambda_{\ell}, \wt{G}(\bx)^{\otimes \ell}\right\rangle\right)\\
=&\left\langle \Lambda_1, I_{2q}-\wt{G}(\bx)\right\rangle+
\sum_{\ell=2}^{k}\left(\left\langle \Lambda_{\ell-1}, \wt{G}(\bx)^{\otimes \ell-1}\right\rangle-\left\langle \Lambda_{\ell}, \wt{G}(\bx)^{\otimes \ell}\right\rangle\right)\\
=&\left\langle \Lambda_1, I_{2q}\right\rangle-\left\langle \Lambda_k, \wt{G}(\bx)^{\otimes k}\right\rangle
=\left\langle \Lambda_k, I_{2q}\right\rangle-\left\langle \Lambda_k, \wt{G}(\bx)^{\otimes k}\right\rangle
=\left\langle \Gamma_k^+, I_{2q}-\wt{G}(\bx)^{\otimes k}\right\rangle.
\end{aligned}
\end{equation}
For each $\ell\in[k]$, equally partition $\Lambda_\ell$ 
as a $2\times 2$ block matrix $[(\Lambda_\ell)_{ij}]_{i,j\in[2]}, $where 
$(\Lambda_\ell)_{ij}$ denotes its $(i,j)$-th block, and let
$\wt{\Lambda}_{\ell}=\diag((\Lambda_\ell)_{22}, (\Lambda_\ell)_{11})$. We have 
$\wt{\Lambda}_{\ell}\succeq 0$ for all $\ell\in[k]$, and 
\[
\left\langle \Lambda_1, I_{2q}-\wt{G}(\bx)\right\rangle=\left\langle \wt{\Lambda}_1, \wt{G}(\bx)\right\rangle,
\ \
\left\langle \Lambda_{\ell}, \left(I_{2q}-\wt{G}(\bx)\right)\otimes \wt{G}(\bx)^{\otimes \ell-1}\right\rangle
=\left\langle \wt{\Lambda}_{\ell}, \wt{G}(\bx)^{\otimes \ell}\right\rangle, 
\]
for all $\ell=2,\ldots, k.$
Therefore, by \eqref{eq::LG}, it holds that
\[
\left\langle \Gamma_k^+, I_{2q}-\wt{G}(\bx)^{\otimes k}\right\rangle=
\sum_{\ell=1}^k \left\langle \wt{\Lambda}_{\ell}, \wt{G}(\bx)^{\otimes \ell}\right\rangle\in\mH(G).
\]
Now we have shown that $1-t\bx^{\ba}\in\mH(G)$ for all $t\in\RR$ satisfying \eqref{eq::t}.

By similar arguments, we can prove that $1+t\bx^{\ba}\in\mH(G)$ for all $t\in\RR$ satisfying
\[
0<t\le \left(\gamma^-+\sum_{k=1}^m\tr{\Gamma_k^-}\right)^{-1}.
\]
Then, by letting 
\[
\lambda_{\ba}=\min\left\{\left(\gamma^++\sum_{k=1}^m\tr{\Gamma_k^+}\right)^{-1}, \left(\gamma^-+\sum_{k=1}^m\tr{\Gamma_k^-}\right)^{-1}\right\},
\]
the proof is complete.
\end{proof}

The following results extends Krivine-Stengle's Positivstellensatz 
\cite{Krivine1964, Stengle1974} (see also \cite{Vas2003}) 
to the PMI setting.

\begin{theorem}\label{th::main2}
    Suppose that Assumption \ref{assump::1} holds. If a polynomial 
    $f(\bx)\in\RR[\bx]$ is positive on $\K$, then $f\in\mH(G)$.
\end{theorem}
\begin{proof}
    As Assumption \ref{assump::1} holds, we have $\RR[\bx]=\mH(G)-\mH(G)$
    and $1$ is an algebraic interior point of $\mH(G)$ by Proposition \ref{prop::unperforated2}.
    Then recall Remark \ref{rmk::core}, we can see that 
    the conclusion follows using arguments analogous to those in the proof of Theorem \ref{th::main1}.
\end{proof}

\subsection{Using sparse structures}
In this part, we assume that the matrix $G(\bx)$ presents
block-diagonal structure, i.e., $G(\bx)=\diag(G_1(\bx), \ldots, G_t(\bx))$ for some $G_i(\bx)\in\bS[\bx]^{q_i}$, 
$q_i\in\N$, $i\in[t]$.
We first give an efficient representation of the cone $\mH(G)$ 
by exploiting $G(\bx)$'s block-diagonal structure. 
Under the additional assumption that each block $G_i(\bx)$ depends on overlapping variable subsets, 
we then derive sparse versions of Positivstellens\"atze Theorems~\ref{th::main1} and \ref{th::main2}.

Denote by $\mathcal{W}(t)$ the set of words $\bw(\bX)$ generated by $t$
noncommuting letters $\bX:=(X_1, \ldots, X_t)$. 
For example, the set of words in 2 noncommuting letters of length 3 is 
\[
\{X_1^3,\ X_1^2X_2,\ X_1X_2X_1,\ X_1X_2^2,\ X_2X_1^2,\ X_2X_1X_2,\ X_2^2X_1,\ X_2^3\}.
\]
For a word $\bw(\bX)$, its evaluation $\bw(\bA)$ on the tuple of matrices 
$\bA:=(A_1,\ldots,A_t)$
is defined as the matrix obtained by substituting 
the letter $X_i$ by $A_i$, $i\in[t]$, 
and forming the corresponding Kronecker product. For example, the evaluations
of words in 2 noncommuting letters of length 3 on $(A_1, A_2)$ are the matrices 
\[
\left\{A_1^{\otimes 3},\ A_1^{\otimes 2}\otimes A_2,\ A_1\otimes A_2\otimes A_1,\ 
A_1\otimes A_2^{\otimes 2},\ A_2\otimes A_1^{\otimes 2},\ A_2\otimes A_1\otimes A_2,
\ A_2^{\otimes 2}\otimes A_1,\ A_2^{\otimes 3}\right\}.
\]
For a word $\bw\in\mathcal{W}(t)$, denote by $w_i$, $i\in[t]$, the frequency of 
the letter $X_i$ appearing in $\bw(\bX)$.
Denote the length of the word $\bw(\bX)$ by $|\bw|$. Thus, it holds $|\bw|=w_1+\cdots+w_t$.

In the following, with a slight abuse of notation, let $\bw(G(\bx))=\bw(G_1(\bx),\ldots,G_t(\bx))$
for $\bw\in\mathcal{W}(t)$ and 
$\bw(\wt{G}(\bx))=\bw(G_1(\bx), I_{q_1}-G_1(\bx),\ldots, I_{q_t}-G_t(\bx))$
for $\bw\in\mathcal{W}(2t)$. 

\begin{proposition}\label{prop::diag}
Suppose that $G(\bx)=\diag(G_1(\bx), \ldots, G_t(\bx))$ for some $G_i(\bx)\in\bS[\bx]^{q_i}$, 
$q_i\in\N$, $i\in[t]$.  If $G(\bx)$ is linear, it holds
\[
\mH(G)=\left\{\lambda_0+\sum_{\bw\in\mathcal{W}(t), |\bw|\in[m]}\left\langle\Gamma_{\bw}, {\bw}(G(\bx))\right\rangle\ \Bigg|\  
    \begin{aligned}
    &m\in\N,\ \lambda_0\in\RR_+,\ \Gamma_{\bw}\in\bS_+^{{d(\bw)}},\\
    & d(\bw)=\prod_{i=1}^t q_i^{w_i},\ |\bw|\in[m]
    \end{aligned}\right\}.
\]
If $G(\bx)$ is nonlinear, it holds
\[
\mH(G)=\left\{\lambda_0+\sum_{\bw\in\mathcal{W}(2t), |\bw|\in[m]}\left\langle\Gamma_{\bw}, {\bw}(\wt{G}(\bx))\right\rangle\ \Bigg|\
    \begin{aligned}
    &m\in\N,\ \lambda_0\in\RR_+,\ \Gamma_{\bw}\in\bS_+^{{e(\bw)}},\\
    & e(\bw)=\prod_{i=1}^t q_i^{w_{2i-1}+w_{2i}},\ |\bw|\in[m]
    \end{aligned}\right\}.
\]
\end{proposition}
%
\begin{proof}
    We only prove the linear case and the nonlinear case follows similarly.
    Let $h(\bx)\in\mH(G)$. By definition, there exist $\lambda_0\in\RR_+$,
    $m\in\N$, $\Lambda_k\in\bS^{q^k}_+$, $q=q_1+\cdots+q_t$, $k\in[m]$, such that 
    \[
    h(\bx)=\lambda_0+\sum_{k=1}^m \left\langle\Lambda_k, G(\bx)^{\otimes k}\right\rangle.
    \]
    Note that for each $k\in[m]$, $G(\bx)^{\otimes k}$ is permutation-similar to a block-diagonal 
    matrix whose blocks consist of 
    \[
    \left\{\bw(G_1(\bx),\ldots,G_t(\bx)),\ \bw\in\mathcal{W},\ |\bw|=k\right\}.
    \]
    For each $\bw\in\mathcal{W}$ with $|\bw|=k$, let $\Gamma_{\bw}$ be the submatrix 
    of $\Lambda_k$ which corresponds to the block  $\bw(G_1(\bx),\ldots,G_t(\bx))$ in $G(\bx)^{\otimes k}$.
    Then, $\Gamma_{\bw}\succeq 0$ and 
    \[
    \sum_{\bw\in\mathcal{W}, |\bw|=k}\left\langle\Gamma_{\bw}, \bw(G_1(\bx),\ldots,G_t(\bx))\right\rangle=
    \left\langle\Lambda_k, G(\bx)^{\otimes k}\right\rangle.
    \]
    Hence, the conclusion follows.
\end{proof}

\begin{remark}
   Proposition \ref{prop::diag} establishes that when $G(\bx)$ is block-diagonal, the PSD
   matrices representing polynomials in $\mH(G)$ can be decomposed into smaller matrices. 
   In Section~\ref{sec::apps}, when we apply Theorems \ref{th::main1} and 
   \ref{th::main2} to PMI-constrained 
   polynomial optimization, this decomposition accelerates SDP solvers for the resulting relaxations.
\qed
\end{remark}

When scalar polynomial blocks appear in $G(\bx)$'s block-diagonal structure, 
the commutativity of scalar multiplication enables reduction in the number 
of PSD matrices required to represent polynomials in $\mH(G)$. Take the linear case for example,

\begin{corollary}\label{cor::bd2}
Suppose that $G(\bx)=\diag(g_1(\bx), \ldots, g_{t-1}(\bx), G_t(\bx))$ for some linear
$g_1(\bx)$, $\ldots$, $g_{t-1}(\bx)\in\RR[\bx]$, $G_t(\bx)\in\bS[\bx]^q$. Then, it holds
\[
\mH(G)=\left\{\lambda_0+
    \sum_{\bg\in\N^t_m}\left\langle\Gamma_{\bg}, g_1(\bx)^{\gamma_1}\cdots g_{t-1}(\bx)^{\gamma_{t-1}}
    G(\bx)^{\otimes \gamma_{t}}\right\rangle\ \Bigg|\  
    \begin{aligned}
    &m\in\N,\ \lambda_0\in\RR_+,\\
    & \Gamma_{\bg}\in\bS_+^{q^{\gamma_t}}, \bg\in\N_m^t
    \end{aligned}\right\}.
\]
\end{corollary}
\begin{proof} The conclusion follows from Proposition \ref{prop::diag}.
\end{proof}

Next, we further assume that the block-diagonal matrix $G(\bx)$ exhibits the correlative sparsity. In other words, 
the variables $\bx$ can be grouped into overlapping subsets 
such that each block $G_i(\bx)$ in $G(\bx)$ depending only on variables from a single subset. Precisely, we consider the following assumption.

\begin{assumption}\label{assump::vs}
The index set $[n]$ can be decomposed as $[n]=\cup_{i=1}^t I_i$ such that
\begin{enumerate}
       \item[(i)] For every $i\in[t]$, $G_i\in\bS[\bx(I_i)]^{q_i}$;
       \item[(ii)] The subsets $\{I_i\}_{i=1}^t$ satisfy the running intersection property (RIP), that is, 
       for every $i\in[t]\setminus\{1\}$, $I_i\cap (\cup_{j=1}^{i-1}I_j)\subseteq I_k$ for some $k\in[i-1]$. 
   \end{enumerate}
\end{assumption}

For each $i\in[t]$, let 
\begin{equation}\label{eq::Ki}
    \K_i:=\{\bx\in\RR^{|I_i|} \mid G_i(\bx)\succeq 0\}.
\end{equation}

\begin{assumption}\label{assump::vs2}
If $G(\bx)$ is linear, each $\K_i$ is compact with nonempty interior. Otherwise, for each $i\in[t]$,
$\K_i$ is compact, $I_{q_i} - G_i(\bx) \succeq 0$ on $\K_i$ and $\RR[\bx(I_i)] = \mH(G_i) - \mH(G_i)$.
\end{assumption}

Now we borrow the idea from \cite{WLT2018} and give the following sparse version of 
the Positivstellens\"atze in Theorems \ref{th::main1} and \ref{th::main2}.

\begin{theorem}\label{th::main3}
    Suppose that Assumptions \ref{assump::vs} and \ref{assump::vs2} hold. Let
    $f=\sum_{i=1}^t f_i\in\RR[\bx]$ with some $f_i\in\RR[\bx(I_i)]$, $i\in[t]$. If $f(\bx)$ is positive on $\K$, then 
    $f\in\sum_{i=1}^t\mH(G_i)$.
\end{theorem}
\begin{proof}
As $f(\bx)>0$ on $\K$, there exists $\varepsilon>0$ such that 
$f(\bx)-\varepsilon>0$ on $\K$.
   By Assumption~\ref{assump::vs2}, for each $i\in[t]$,
   there exists $M_i>0$ such that 
   $g_i(\bx(I_i)):=M_i-\sum_{j\in I_i}x_j^2\ge 0$ on $\K_i$. 
   Let $G_i'=\diag(G_i, g_i)$, then $G'_i(\bx(I_i))\succeq 0$ defines the same set $\K_i$. As $f(\bx)-\varepsilon>0$ on $\K$, by \cite[Theorem 1]{KM2009}, for each $i\in[t]$, there exist 
   SOS polynomial $\sigma_i\in\RR[\bx(I_i)]$ and SOS polynomial matrix $\Sigma_i\in\bS[\bx(I_i)]^{q_i+1}$ such that 
   \[
   f(\bx)-\varepsilon=\sum_{i=1}^t \left(\sigma(\bx(I_i))
   +\left\langle\Sigma_i(\bx(I_i)), G'_i(\bx(I_i))\right\rangle\right).
   \]
   Let \[
   f_i(\bx(I_i)):=\frac{\varepsilon}{t}+\sigma_i(\bx(I_i))
   +\left\langle\Sigma_i(\bx(I_i)), G'_i(\bx(I_i))\right\rangle,\quad
   i=1,\ldots,t.
   \]
   It is clear that each $f_i(\bx(I_i))$ is positive on $\K_i$.
   Then, by Assumption \ref{assump::vs2} and Theorems \ref{th::main1} and \ref{th::main2}, we have $f_i\in\mH(G_i)$ 
   and the conclusion follows.
\end{proof}

\begin{remark}\label{rmk::sparse}
Compared with Theorems \ref{th::main1} and \ref{th::main2}, 
if the correlative sparsity condition in Theorem \ref{th::main3} is satisfied 
by $f(\bx)$ and $G(\bx)$, the number of PSD matrices required 
to represent $f(\bx)$ can be greatly reduced. 
For instance, in the linear case and under the simplification $q_1=\cdots=q_t=q$, 
representing $f(\bx)$ using Theorem \ref{th::main1} and Proposition \ref{prop::diag} 
requires $t^k$ PSD matrices of size $q^k$ for each $k\in[m]$. 
In contrast, Theorem~\ref{th::main3} requires only $t$ such matrices for each $k\in[m]$.   
This improvement can significantly enhance the scalability and computational tractability 
of the SDP relaxations developed in Section \ref{sec::apps}
for PMI-constrained polynomial optimization that exhibit correlative 
sparsity (see Remark~\ref{rmk::sparse}).
\end{remark}

%

\section{Applications to polynomial optimization with matrix inequality constraints}\label{sec::apps}

Consider the problem of minimizing a polynomial function over a semialgebraic set 
defined by a polynomial matrix inequality:
\begin{flalign*}
\text{($\mathbb{P}$)} &&& f^{\star}:=\min_{\bx\in\K} f(\bx)\ \ \text{s.t. } \ \K=\{\bx\in\RR^n \mid G(\bx)\succeq 0\}, &&
\end{flalign*}
where $f(\bx)\in\RR[\bx]$ and $G(\bx)\in\bS[\bx]^q$.

In the definition of $\mH(G)$, by fixing the number $m\in\N$,  
we obtain the $m$-th truncation of $\mH(G)$ and denote it by $\mH_m(G)$.
Then, we have $\mH(G)=\cup_{m=0}^{\infty}\mH_m(G)$. 
For any $m\in\N$, consider the problem
\begin{flalign*}
\text{($\mathbb{H}_m$)} &&& f^{\star}_m:=\sup_{r\in\RR}\ r \ \ \text{s.t. }\  f(\bx)-r\in\mH_m(G). &&
\end{flalign*}

By the definition of $\mH_m(G)$, the problem ($\mathbb{H}_m$) is a semidefinite program.
\begin{remark}{\rm
    When $G(\bx)$ is diagonal, the SDP relaxation ($\mathbb{H}_m$) retrieves the LP-relaxation 
    for scalar polynomial optimization problems obtained by Lasserre in \cite{LasserreLP2002,LasserreLP2005}.}
\end{remark}

\subsection{Asymptotic behavior of the SDP relaxations $\{(\mathbb{H}_m)\}$}
We establish the asymptotic convergence of the optimal values of the SDP relaxations $\{(\mathbb{H}_m)\}$
as $m\to\infty$.
\begin{theorem}
    Suppose that $G(\bx)$ is linear (resp., nonlinear) and Assumption \ref{assump::0}
    (resp., \ref{assump::1}) holds. Then, $f_m^{\star} \uparrow f^{\star}$ as $m\to\infty$. 
\end{theorem}
\begin{proof}
    As $\mH_m(G)\subseteq\mH_{m+1}(G)$, the sequence $\{f_m^{\star}\}_{m\in\N}$ is nondecreasing. 
    By definition, all elements in $\mH_m(G)$ are nonnegative on $\K$. Then, 
    $f(\bx)-f^{\star}_m\ge 0$ on $\K$ and hence $f^{\star}\ge f^{\star}_m$ for all $m\in\N$.
    Note that for any $\varepsilon>0$, $f(\bx)-f^{\star}+\varepsilon>0$ on $\K$. 
    Since $G(\bx)$ is linear (resp., nonlinear) and Assumption \ref{assump::0} 
    (resp., Assumption \ref{assump::1}) 
    holds, Theorem \ref{th::main1} (resp., Theorem \ref{th::main2}) implies that 
    $f^{\star}_m\ge f^{\star}-\varepsilon$ for some $m\in\N$. Thus, it holds 
    $f_m^{\star} \uparrow f^{\star}$ as $m\to\infty$. 
\end{proof}

Let $m_0\in\N$ be a number such that the SDP relaxation ($\mathbb{H}_m$) is feasible for all $m\ge m_0$.
As with the LP-relaxation for scalar polynomial optimization \cite[Theorem 3.1]{LasserreLP2005},
we analyze the asymptotic behavior of the optimal solutions of $\{(\mathbb{H}_m)\}$ as $m\to\infty$ 
in the linear case. The result extends to the nonlinear case via analogous arguments 
and we omit the details for simplicity. 
\begin{proposition}
    Suppose that $G(\bx)$ is linear and $G(\bu)\succ 0$ for some $\bu\in\K$.
    For any $m\ge m_0$, let 
    \[
    \left(r_m, \lambda_{m,0}\in\RR_+, \  \Lambda_{m,k}\in\bS^{q^k}_+,\ k\in[m]\right)
    \]
    be an  optimal solution of $(\mathbb{H}_m)$. Then the follow statements hold true.
    \begin{enumerate}
        \item[\upshape (i)] There exist a subsequence $\{m_j\}\subset\N$ and an infinite sequence of 
        PSD matrices 
        $\Lambda_{\star,k}\subset\bS^{q^k}_+,\ k\in\N$, such 
        that as $j\to\infty$,
        \[
        \lambda_{m_j,0}\to 0, \ \Lambda_{m_j,k}\to \Lambda_{\star,k}, \ \forall k\in\N,
        \]
        and for any $\bx\in\K$, 
        \[
        f(\bx)-f^{\star}\ge 
        \sum_{k\in\N} \left\langle\Lambda_{\star, k}, G(\bx)^{\otimes k}\right\rangle.
        \]
        \item[\upshape (ii)] Let $\bu^{\star}$ be a minimizer of $(\mathbb{P})$. For any decomposition  
        $\Lambda_{\star, k}=\sum_{\ell=1}^{s_k} \bv^{(k,\ell)}(\bv^{(k,\ell)})^{\intercal}$
        with $\bv^{(k,\ell)}\in\RR^{q^k}$, $k\in\N$, 
        it holds that $G(\bu^{\star})^{\otimes k} \bv^{(k,\ell)}=0$ for all $\ell\in[s_k]$, $k\in\N$. 
    \end{enumerate}
\end{proposition}
\begin{proof}
    As $G(\bu)\succ 0$, $G(\bu)^{\otimes k}\succ 0$ for all $k\in\N$. For any $m\ge m_0$, it holds 
    \begin{equation}\label{eq::rm}
    f(\bu)-r_m=\lambda_{m,0}+\sum_{k\in [m]} \left\langle\Lambda_{m, k}, G(\bu)^{\otimes k}\right\rangle.
    \end{equation}
    As $\{r_m\}_{m\ge m_0}$ is nondecreasing, we have 
    \[
    \lambda_{m,0}\le f(\bu)-r_{m}\le f(\bu)-r_{m_0}\quad \text{and}\quad
    \left\langle\Lambda_{m, k}, G(\bu)^{\otimes k}\right\rangle \le f(\bu)-r_{m}\le f(\bu)-r_{m_0}
    \]
    for all $m\ge m_0$ and $k\in[m]$.
    In particular, the sequence $\{\lambda_{m,0}\}_{m\ge m_0}$ is bounded. 
    For each $k\in\N$, denote by $\lambda_{\min}^{(k)}(\bu)$ the smallest eigenvalue of $G(\bu)^{\otimes k}$.
    Then, $\lambda_{\min}^{(k)}(\bu)>0$ for all $k\in\N$ and 
    \[
    f(\bu)-r_{m_0}\ge \left\langle\Lambda_{m, k}, G(\bu)^{\otimes k}\right\rangle\ge 
    \left\langle\Lambda_{m, k}, \lambda_{\min}^{(k)}(\bu) I_{q^k}\right\rangle
    =\lambda_{\min}^{(k)}(\bu)\tr{\Lambda_{m, k}}. 
    \]
    For each $k\in\N$, as $\Lambda_{m, k}\succeq 0$, the matrices $\{\Lambda_{m, k}\}_{m\ge m_0}$
    (let $\Lambda_{m, k}=0$ if $m<k$)
    are entry-wise bounded. Denote by $\ve(\Lambda_{m, k})$ the row vectorization of the matrix
    $\Lambda_{m, k}$. For each $m\ge m_0$, denote by $\bw_m$ the element of the space $l_{\infty}$ obtained by 
    completing the vector $(\lambda_{m,0}, \ve(\Lambda_{m, k}), k\in[m])$ with infinitely many zeros. 
    Since the vectors in $\{\bw_m\}_{m\ge m_0}$ are entry-wise bounded, by Tychonoff's theorem, 
    there is a convergent subsequence $\{\bw_{m_j}\}_{j\in\N}$. From the limiting point, we
    extract a real number $\lambda_{\star,0}$ and matrices $\Lambda_{\star, k}\in\bS^{q^k}$ such that 
    $\lambda_{m_j,0}\to \lambda_{\star, 0}$ and $\Lambda_{m_j,k}\to \Lambda_{\star,k}$
    for all $k\in\N$. By the entry-wise convergence, we have $\lambda_{\star,0}\ge 0$ and 
    $\Lambda_{\star,k}\succeq 0$ for all $k\in\N$. As Assumption \ref{assump::0} holds, 
    $\lim_{j\to\infty}r_{m_j}=f^{\star}$ by Theorem \ref{th::main1}. Then,
    for any $\bx\in\K$, by Fatou's lemma, we have 
    \[
    f(\bx)-f^{\star}=\lim_{j\to\infty}(f(\bx)-r_{m_j})
    =\liminf_{j\to\infty}(f(\bx)-r_{m_j})\ge 
    \lambda_{\star,0}+
    \sum_{k\in \N} \left\langle\Lambda_{\star, k}, G(\bx)^{\otimes k}\right\rangle. 
    \]
    Let $\bu^{\star}$ be a minimizer of ($\mathbb{P}$). Then, 
    \[
        0=f(\bu^{\star})-f^{\star}\ge 
        \lambda_{\star,0}+
    \sum_{k\in \N} \left\langle\Lambda_{\star, k}, G(\bu^{\star})^{\otimes k}\right\rangle.
        \]
    As $\lambda_{\star,0}\ge 0$ and  $\Lambda_{\star,k}\succeq 0$ for all $k\in\N$, 
    we must have $\lambda_{\star,0}=0$ and 
    $\left\langle\Lambda_{\star, k}, G(\bu^{\star})^{\otimes k}\right\rangle= 0$ for all $k\in\N$.
    This completes the proof of (i) and (ii).
\end{proof}

\subsection{Exactness of the SDP relaxations $\{(\mathbb{H}_m)\}$}
We say that the relaxations \{($\mathbb{H}_m$)\} are exact if $f_m^{\star}=f^{\star}$ for some $m\in\N$.
Similar to the LP-relaxation for scalar polynomial optimization \cite[Proposition 3.2]{LasserreLP2005},
we will show that the SDP relaxations $\{(\mathbb{H}_m)\}$ can not be exact in general.
\begin{lemma}\label{lem::AB}
    For $A, B\in\RR^{s\times t}$, if the null space $\nul(A)\subseteq\nul(B)$, then  
    $\nul(A^{\otimes k})\subseteq\nul(B^{\otimes k})$ for any $k\in\N$.
\end{lemma}
\begin{proof}
For any $k\in\N$, we have (see \cite[14.8]{SM2013})
    \[
    \begin{aligned}
        \nul(A^{\otimes k})&=\nul(A)\otimes \RR^t \otimes\cdots\otimes \RR^t+
    \RR^t\otimes\nul(A)\otimes\cdots\otimes\RR^t+\RR^t\otimes\RR^t\otimes\cdots
    \otimes \nul(A),\\
    \nul(B^{\otimes k})&=\nul(B)\otimes \RR^t \otimes\cdots\otimes \RR^t+
    \RR^t\otimes\nul(B)\otimes\cdots\otimes\RR^t+\RR^t\otimes\RR^t\otimes\cdots
    \otimes \nul(B).
    \end{aligned}
    \]
    As $\nul(A)\subseteq\nul(B)$, it is clear that $\nul(A^{\otimes k})\subseteq\nul(B^{\otimes k})$.
\end{proof}

\begin{proposition}\label{prop::notexact}
Let $\bu^{\star}$ be a minimizer of $(\mathbb{P})$ and $G(\bx)$ be linear (resp., nonlinear). 
If there exists a nonoptimal point $\bu\in\K$ of such that 
$\nul(G(\bu^{\star}))\subseteq \nul(G(\bu))$ (resp., $\nul(\wt{G}(\bu^{\star}))\subseteq \nul(\wt{G}(\bu))$), 
then no relaxation $(\mathbb{H}_m)$ is exact.
In particular, if $G(\bu^{\star})\succ 0$ (resp., $G(\bu^{\star})\succ 0$ and $I_q-G(\bu^{\star})\succ 0$), 
then no relaxation $(\mathbb{H}_m)$ is exact.
\end{proposition}
\begin{proof}
We only need to consider the case when $G(\bx)$ is linear.
Suppose to the contrary that the $m$-th relaxation \{($\mathbb{H}_m$)\} is exact, i.e., the
representation
\[
f(\bx)-f^{\star}=\lambda_{m,0}+\sum_{k\in [m]} \left\langle\Lambda_{m, k}, G(\bx)^{\otimes k}\right\rangle,
\]
holds for some $\lambda_{m,0}\in\RR_+$ and $\Lambda_{m, k}\in\bS^{q^k}_+$, $k\in[m]$. Then,
\[
0=f(\bu^{\star})-f^{\star}=\lambda_{m,0}+\sum_{k\in [m]} \left\langle\Lambda_{m, k}, G(\bu^{\star})^{\otimes k}\right\rangle,
\]
which implies that $\lambda_{m,0}=0$ and $\left\langle\Lambda_{m, k}, G(\bu^{\star})^{\otimes k}\right\rangle$
for all $k\in[m]$. For any decomposition  
$\Lambda_{m, k}=\sum_{\ell=1}^{s_k} \bv^{(k,\ell)}(\bv^{(k,\ell)})^{\intercal}$ with 
$\bv^{(k,\ell)}\in\RR^{q^k}$, it holds that 
$G(\bu^{\star})^{\otimes k} \bv^{(k,\ell)}=0$ for all $\ell\in[s_k]$, $k\in[m]$. 
Since $\nul(G(\bu^{\star}))\subseteq \nul(G(\bu))$, by Lemma \ref{lem::AB}, we have 
$\nul(G(\bu^{\star})^{\otimes k})\subseteq \nul(G(\bu)^{\otimes k})$ for all $k\in\N$.
Thus, $G(\bu)^{\otimes k} \bv^{(k,\ell)}=0$ for all $\ell\in[s_k]$, $k\in[m]$, and hence 
$\left\langle\Lambda_{m, k}, G(\bu)^{\otimes k}\right\rangle=0$
for all $k\in[m]$. Therefore, as $\bu\in\K$ is nonoptimal, we have
\[
0<f(\bu)-f^{\star}=\lambda_{m,0}+\sum_{k\in [m]} \left\langle\Lambda_{m, k}, G(\bu)^{\otimes k}\right\rangle=0,
\]
a contradiction. 
\end{proof}

\begin{remark}\label{rmk::exact}
When $G(\bx)$ is linear, the null space $\nul(G(\bx))$ is constant over the relative interior 
of any face of $\K$ (\cite[Corollary 1]{RG1995}).  
Thus when $G(\bx)$ is linear, Proposition \ref{prop::notexact} implies that the only case where the relaxations 
\{($\mathbb{H}_m$)\} may be exact is if every global minimizer of ($\mathbb{P}$) is a 
$0$-dimensional face of $\K$ or if $f(\bx)-f^{\star}$ vanishes on every face of $\K$ that 
contains a global minimizer of ($\mathbb{P}$) as a relative interior. 
\end{remark}

In the following examples, we implement the relaxations \{($\mathbb{H}_m$)\} using {\sf Yalmip} \cite{YALMIP}
and solve the resulting SDPs by {\sf Mosek} \cite{mosek}. 
To formulate ($\mathbb{H}_m$) as semidefinite program, 
we need to express $f(\bx)-r$ as a polynomial in $\mH_m(G)$. To state the identity of the two polynomials, 
it is inefficient to match their coefficients since expanding a polynomial in 
$\mH_m(G)$ symbolically to get the coefficients with respect to the monomial basis could be 
computationally prohibitive. Instead, if $G(\bx)$ is linear (resp., nonlinear), 
we equate their values on $\binom{n+k}{n}$ $\left(\text{resp.}, \binom{n+k\cdot\deg(G)}{n}\right)$ generic 
points randomly generated on the box $[-1, 1]^n$. 
All computations in the sequel were performed and timed on a desktop 
with 8-Core Intel i7 3.8 GHz CPUs and 8 GB RAM.

\begin{example}\label{ex::1}{
Consider the optimization problem with linear PMI constraint
\begin{equation}\label{eq::P1}
f^{\star}:=\min_{\bx\in\K} f(\bx)\quad\text{s.t. }\ \K=\left\{\bx\in\RR^3 \colon 
\left[
\begin{array}{cc}
  x_3+x_2   & 2x_3-x_1 \\
   2x_3-x_1  &  x_3-x_2
\end{array}
\right]\succeq 0, \ x_3\le 1\right\}. 
\end{equation}
The spectrahedron $\K$ is depicted in Figure \ref{fig::1}. As shown, the faces of $\K$ 
come in dimensions $0$, $1$, $2$ and $3$. The $0$-dimensional faces consist of the origin and 
the points on the circle defined by $(x_1-2)^2+x_2^2=1$ and $x_3=1$ (the blue circle in Figure \ref{fig::1}).
The segments connecting the origin and the points on this circle 
(the red edges in Figure \ref{fig::1}) form the $1$-dimensional faces of $\K$. 
The $2$-dimensional face of $\K$ is the planar disk bounded by the blue circle and the 
$3$-dimensional face of $\K$ is itself.

Clearly, Assumption \ref{assump::0} holds for $\K$.
By Remark \ref{rmk::exact}, the only case where the relaxations 
\{($\mathbb{H}_m$)\} may be exact is if every global minimizer of \eqref{eq::P1} is either the origin
or a point on the blue circle,  or if $f(\bx)-f^{\star}$ vanishes on the edge of $\K$ that 
contains a global minimizer as a relative interior. 

Now we consider four objective functions 
\[
\begin{aligned}
&f_1(\bx)=\left(x_1+\frac{1}{2}\right)^2+x_2^2+\left(x_3-\frac{3}{2}\right)^2,\
&&f_2(\bx)=(x_1-2)^2+x_2^2+(x_3-1-\sqrt{2})^2,\\
&f_3(\bx)=(x_1+1)^2+x_2^2+(x_3-1)^2, \ 
&&f_4(\bx)=x_1^2+x_2^2+(x_3-2)^2.
\end{aligned}
\]
It is easy to see that the functions $f_1,\ldots,f_4$ have the same minimum $f^{\star}=2$ over $\K$, 
which is achieved at their respective unique minimizers
$\bu^{(1)}=(1/2, 0, 1/2)$, $\bu^{(2)}=(2, 0, 1)$, $\bu^{(3)}=(0, 0, 0)$, $\bu^{(4)}=(1, 0, 1)$.

As $\bu^{(1)}$ and $\bu^{(2)}$ lie in the relative interior of the corresponding faces of $\K$,
the relaxations \{($\mathbb{H}_m$)\} cannot be exact for $f_1$ and $f_2$. In fact, for $f_1$, we
obtain the lower bounds
\[
f^{\star}_2=1.5000,\ f^{\star}_3=1.8333,\ f^{\star}_4=1.8333,\ f^{\star}_5=1.9000,
f^{\star}_6=1.9000,\ f^{\star}_7=1.9286.
\]
For $f_2$, we obtain the lower bounds
\[
f^{\star}_2=1.0000,\ f^{\star}_3=1.5000,\ f^{\star}_4=1.6667,\ f^{\star}_5=1.7500,
f^{\star}_6=1.8019,\ f^{\star}_7=1.9657.
\]

Since $\bu^{(3)}$ and $\bu^{(4)}$ are $0$-dimensional faces of $\K$, 
the relaxations \{($\mathbb{H}_m$)\} may be exact for $f_3$ and $f_4$. The numerical results suggest that
it is indeed the case. In fact, we have $f_2^{\star}=2.0000$ for both $f_3$ and $f_4$. 

    \begin{figure}\label{fig::1}
\centering
\scalebox{0.5}{
\includegraphics[trim=10 230 10 230,clip]{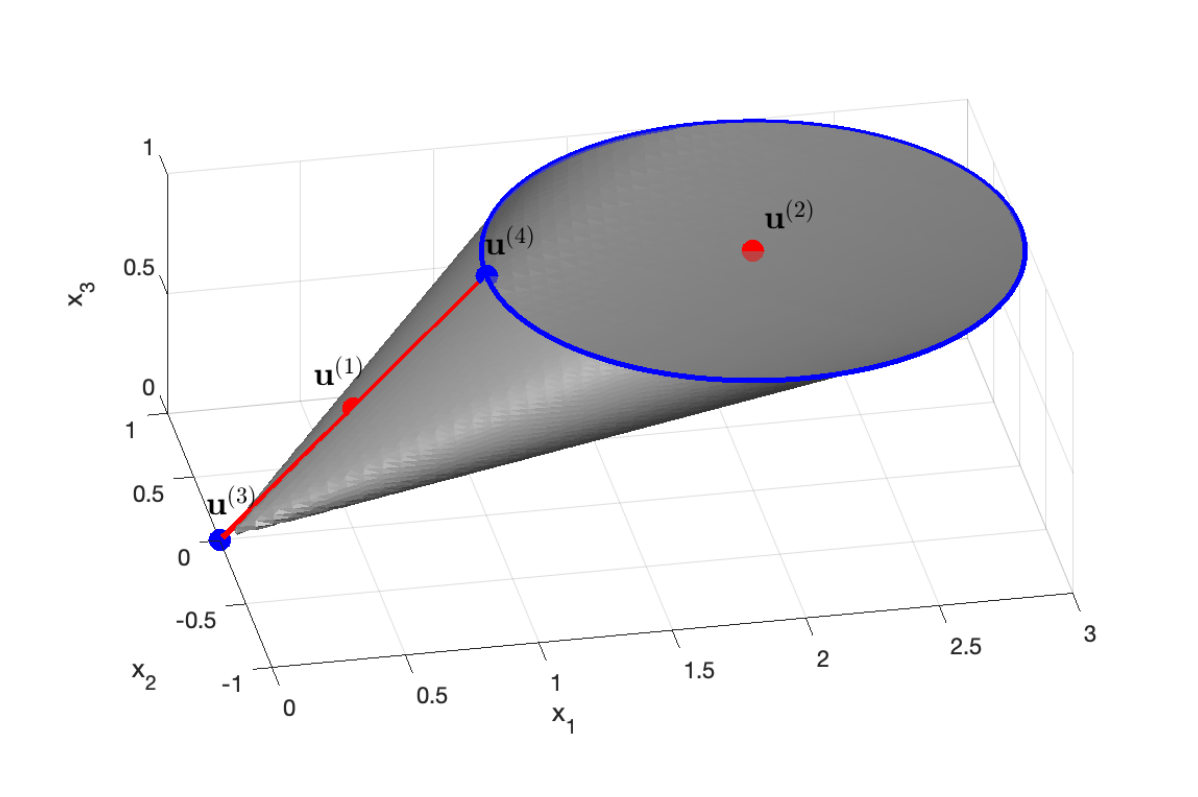}}
\caption{The feasible set $\K$ of Problem \eqref{eq::P1}.}
\end{figure}}
\end{example}

\begin{example}{\rm
    Consider the optimization problem with nonlinear PMI constraint
\begin{equation}\label{eq::P2}
f^{\star}:=\min_{\bx\in\K} f(\bx)\quad\text{s.t. }\ \K=\left\{\bx\in\RR^2 \colon 
H(\bx):=\left[
\begin{array}{cc}
  x_2-\frac{x_1}{2}   & x_1x_2 \\
   x_1x_2  &  2x_1-x_2
\end{array}
\right]\succeq 0 \right\}. 
\end{equation}
Clearly, a point $\bx\in\RR^2$ is feasible if and only if it satisfies 
\[
g_1(\bx)=x_2-\frac{x_1}{2}\ge 0, \quad g_2(\bx)=2x_1-x_2\ge 0, \quad
g_3(\bx)=\frac{5}{2}x_1x_2-x_1^2-x_2^2-x_1^2x_2^2\ge 0.
\]
In fact, it is easy to see that the set $\K$ is the compact region bounded the 
curve $g_3(\bx)=0$ in  the first quadrant, as shown in Figure \ref{fig::2}. 

Note that the sum of two eigenvalues of $H(\bx)$ over $\K$ is bounded by the trace $\frac{3x_1}{2}$. 
Since $\frac{\partial g_3}{\partial x_2}(\bv)=0$ where $\bv=\left(\frac{3}{4}, \frac{3}{5}\right)$, 
the maximum of $\frac{3x_1}{2}$ over $\K$ is $\frac{9}{8}$ achieved at $\bv$.
As $\lambda_{\max}(H(\bv))=\frac{9}{8}$,
we can see that the largest eigenvalue of $H(\bx)$ over $\K$ is $\frac{9}{8}$.
Let $G(\bx)=\frac{8}{9}H(\bx)$, then $G(\bx)\succeq 0$ defines the same set $\K$ and 
Assumption \ref{assump::1} is satisfied. 

Consider the objective function 
\[
f(x_1, x_2)=x_1^2(x_1-1)^2+x_2^2(x_2-1)^2+(x_1-x_2)^2.
\]
Since the Hessian matrix of $f$ is indefinite at $\bu^{(0)}=\left(\frac{1}{2}, \frac{1}{2}\right)\in\K$,
$f$ is a nonconvex polynomial. Clearly, the global minimum of $f(\bx)$ over $\RR^2$ is 0,
achieved at $\bu^{(1)}=(0, 0)\in\K$ and $\bu^{(2)}=(1, 1)\not\in\K$. Hence, 
$f^{\star}=0$ and $\bu^{(1)}$ is the only minimizer of Problem \eqref{eq::P2}. 
The numerical result $f_4^{\star}=3.0415\times 10^{-12}$ suggests that 
the $4$-th relaxation $(\mathbb{H}_4)$ is exact. 

Next we replace the variables $\bx$ by $2\bx$ in the objective function and scale it by 
$\frac{1}{4}$.  That is, we let
\[
f(x_1, x_2)=x_1^2(2x_1-1)^2+x_2^2(2x_2-1)^2+(x_1-x_2)^2.
\]
Obviously, the optimal value remains $f^{\star}=0$, attained at two minimizers $\bu^{(0)}$
and $\bu^{(1)}$. Notice that $G(\bu^{(0)})\succ 0$ and $I_2-G(\bu^{(0)})\succ 0$. Then, 
by Proposition \ref{prop::notexact}, no relaxation ($\mathbb{H}_m$) is exact,  which is 
consistent with the numerical results
$f_4^{\star}=-0.0478$, $f_5^{\star}=-0.0274$ and $f_6^{\star}=-0.0195$.
\qed
    \begin{figure}\label{fig::2}
\centering
\scalebox{0.5}{
\includegraphics[trim=10 220 10 230,clip]{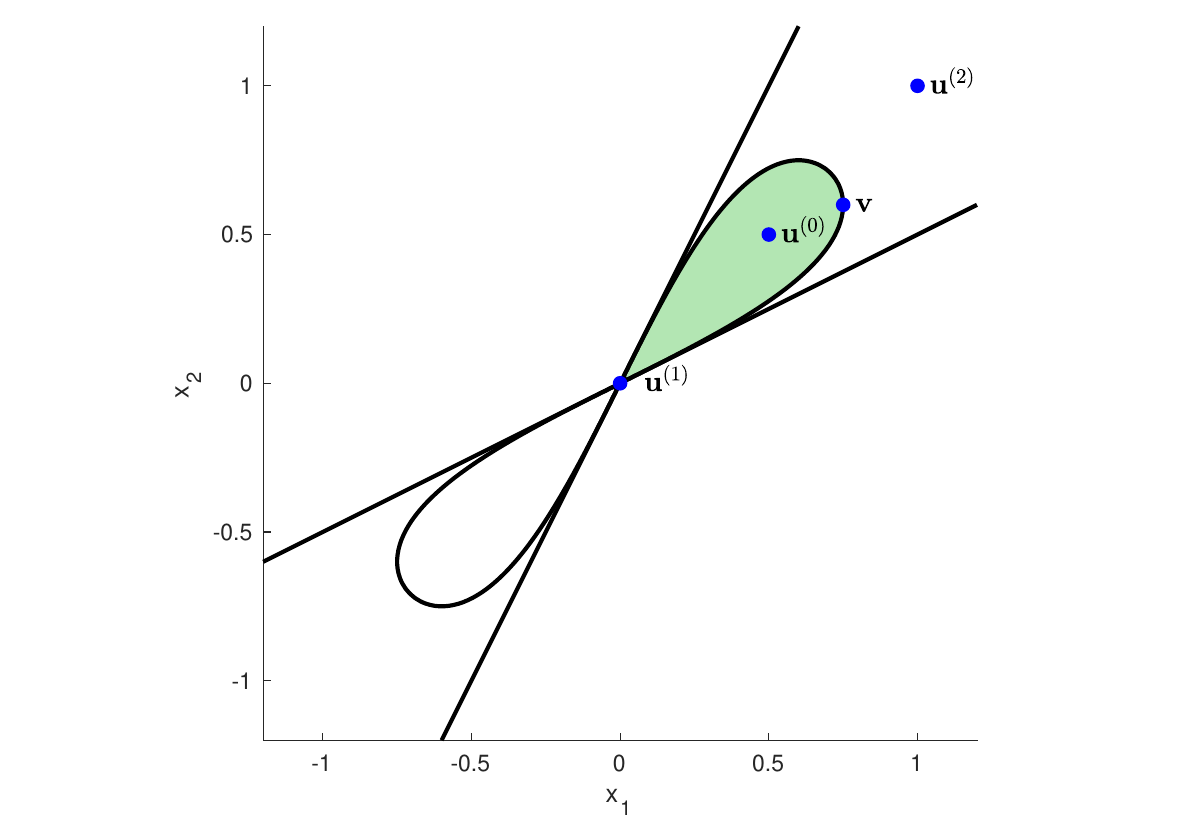}}
\caption{The feasible set $\K$ of Problem \eqref{eq::P2}.}
\end{figure}
}
\end{example}

\subsection{Sparse SDP relaxations}
    If the matrix $G(\bx)$ exhibits a block-diagonal structure, i.e., it holds that
    $G(\bx)=\diag(G_1(\bx), \ldots, G_t(\bx))$ for some $G_i(\bx)\in\bS[\bx]^{q_i}$, $q_i\in\N$, 
    $i\in[t]$, one can leverage the characterization of $\mathcal{H}(G)$ provided in Proposition 
    \ref{prop::diag}. Applying this to the relaxations $(\mathbb{H}_m)$ 
    allows the PSD matrices in the associated SDP problems 
    to be decomposed into smaller matrices, which reduces the computational burden. 
    
    Furthermore, if each monomial in $f(\bx)$ and each block in $G(\bx)$ 
    depends on overlapping variable subsets, then one can construct the sparse relaxations
    \begin{flalign*}
    \text{($\mathbb{H}^{\sps}_m$)} &&& f^{\sps}_m:=\sup_{r\in\RR}\ r \ \ 
    \text{s.t. }\  f(\bx)-r\in\sum_{i=1}^t\mH_m(G_i). &&
    \end{flalign*}
    
The asymptotic convergence of the optimal values $\{f^{\sps}_m\}$ as $m\to\infty$
can be established using Theorem \ref{th::main3}.
\begin{theorem}\label{th::sparseSDP}
    Suppose that Assumptions \ref{assump::vs} and \ref{assump::vs2} hold. Let
    $f=\sum_{i=1}^t f_i\in\RR[\bx]$ with some $f_i\in\RR[\bx(I_i)]$, $i\in[t]$. 
    Then, $f_m^{\sps} \uparrow f^{\star}$ as $m\to\infty$.     
\end{theorem}
    
     The asymptotic and exactness properties obtained for $(\mathbb{H}_m)$ extend analogously 
     to $(\mathbb{H}^{\sps}_m)$. We omit the details for simplicity.
     
     As mentioned in Remark \ref{rmk::sparse}, the resulting SDP problem for $(\mathbb{H}^{\sps}_m)$ 
     is of a much smaller size than that of $(\mathbb{H}_m)$. 
     This key advantage allows $(\mathbb{H}^{\sps}_m)$ to be applied to problems of 
     a significantly larger scale that possess the requisite sparse structure.

\begin{example}\label{ex::4}{\rm
    Consider the optimization problem  ($\mathbb{P}$) where 
    \[
    f(\bx)=\sum_{i=1}^n(x_i-1)^2,\quad G(\bx)=\diag(G_1(x_1, x_2),\ldots, G_{n-1}(x_{n-1}, x_n)),
    \]
    with 
    \[
    G_i(x_i, x_{i+1})=\left[
\begin{array}{cc}
  1-x_i   & x_{i+1} \\
   x_{i+1} &  1+x_i
\end{array}
\right],\quad i=1,\ldots,n-1.
    \]
    Equavilently, we have 
    \[
    \K=\{\bx\in\RR^n \colon 1-x_i^2-x_{i+1}^2\ge 0, \ i=1,\ldots,n-1\}.
    \]
    Since problem ($\mathbb{P}$) is convex, verification of the Karush–Kuhn–Tucker (KKT) optimality conditions yields the following results:
    \begin{enumerate}
        \item[(1)] For even $n$, the optimal value is $f^{\star}=\frac{n}{2}(3-2\sqrt{2})$, attained at 
        the minimizer $\bx^{\star}=\left(\frac{\sqrt{2}}{2},\ldots, \frac{\sqrt{2}}{2}\right)$;
        \item[(2)] For odd $n$, the optimal value is $f^{\star}=\frac{n+1}{2}(a-1)^2+\frac{n-1}{2}(b-1)^2$,
        attained at the minimizer 
        $\bx^{\star}=\left(a, b, a, b, \ldots, a, b, a\right)$,
        where $a$ and $b$ are positive and satisfy
        \[
        a=\frac{(n+1)b}{2b+n-1},\quad a^2+b^2=1.
        \]
    \end{enumerate}
    
Now by increasing the number $n$, we compare the computational performance of two relaxations  
$(\mathbb{H}_m)$ and $(\mathbb{H}^{\sps}_m)$. 
For $(\mathbb{H}_m)$, we use the characterization of $\mH(G)$ 
in Proposition~\ref{prop::diag} to reduce the computational burden. 
Numerically, we find that both of the relaxation $(\mathbb{H}_m)$ 
and $(\mathbb{H}^{\sps}_m)$ are exact for ($\mathbb{P}$) at the order $m=4$. 
Numerical results for $(\mathbb{H}_4)$ 
and $(\mathbb{H}^{\sps}_4)$ are presented in Table \ref{tab::2}, 
where the symbol `$-/-$' indicates that 
{\tt Mosek} runs out of memory. 

    \begin{table}[htb]\caption{Numerical results for the problem in Example \ref{ex::4}.}
\label{tab::2}
\centering
\begin{tabular}{cccccc} 
\midrule[0.8pt]
$(n, f^{\star})$  & $(3, 0.2367)$ & $(4, 0.3431)$ & $(5, 0.4167)$ & $(6, 0.5147)$& $(7, 0.5918)$\\
\midrule[0.4pt]
$f^{\star}_4/\text{time}$&$0.2367$/2.12s &$0.3431/$8.69s &$0.4167/$29.2s
&$0.5147/$106s &$0.5918/338s$\\
$f^{\sps}_4/\text{time}$&$0.2367/$0.90s &$0.3431/$0.97s &$0.4167/$1.10s
&$0.5147/$1.37s &$0.5918/$1.64s\\
\midrule[0.8pt]
\end{tabular}

\begin{tabular}{cccccc} 
\midrule[0.8pt]
$(n, f^{\star})$  & $(8, 0.6863)$ & $(9, 0.7653)$ & $(10, 0.8579)$ & $(15, 1.2827)$& $(20, 1.7157)$\\
\midrule[0.4pt]
$f^{\star}_4/\text{time}$&$0.6863$/54.5m &$-/-$ &$-/-$
&$-/-$ &$-/-$\\
$f^{\sps}_4/\text{time}$&$0.6863/$2.31s &$0.7653/$3.12s &$0.8579/$4.77s &$1.2828/$17.2s &$1.7158/$88.3s\\
\midrule[0.8pt]
\end{tabular}
\end{table}
    }
\end{example}

\subsection{Comparing with SOS-based SDP relaxations}

Applying the SOS-based Positivstellensatz in Theorem \ref{th::psatz} 
to the problem ($\mathbb{P}$), 
one can construct the following $m$-th SDP relaxation for ($\mathbb{P}$)
\begin{flalign*}
\text{($\mathbb{P}_m$)} &&& \left\{
\begin{aligned}
    f^{\sos}_m:=\sup_{r,\sigma,\Sigma}&\ r \\
    \text{s.t.}&\  f(\bx)-r=\sigma(\bx)+\left\langle \Sigma(\bx), G(\bx)\right\rangle,\\
    &\ r\in\RR,\ \sigma\in\RR[\bx]\ \text{and}\ \Sigma\in\bS[\bx]^q \text{ are SOS},\\
    &\ \deg(\sigma), \deg(\left\langle \Sigma(\bx), G(\bx)\right\rangle)\le 2m.
    \end{aligned}
    \right. &&
\end{flalign*}
Under Assumption \ref{assump::archi}, by Theorem \ref{th::psatz}, we have 
$f^{\sos}_m\uparrow f^{\star}$ as $m\to\infty$. 

Now we 
compare the size of the SDP relaxations of 
$(\mathbb{H}_{2m})$ and ($\mathbb{P}_m$), 
which use the same degree bound $2m$ in the representation of $f(\bx)-r$.
By equating two polynomials of degree at most $2m$, 
both $(\mathbb{H}_{2m})$ and ($\mathbb{P}_m$) have the same number $\binom{n+2m}{n}$ of 
equality constraints on their respective variables. 
The variables in the relaxation $(\mathbb{H}_{2m})$ are the nonnegative number $\lambda_0$
and $2m$ PSD matrices $\Lambda_k$, $k\in[2m]$. For each $k\in[2m]$, 
the size of $\Lambda_k$ is $q^k$ when $G(\bx)$ is linear and 
$(2q)^k$ when $G(\bx)$ is nonlinear. In comparison, the 
variables in the relaxation ($\mathbb{P}_m$) involve two PSD matrices of size 
$\binom{n+m}{n}$ and $q\binom{n+m}{n}$, associated with the SOS polynomial $\sigma(\bx)$ and 
the SOS polynomial matrix $\Sigma(\bx)$, respectively.  

The size of PSD matrices in both $(\mathbb{H}_{2m})$ and ($\mathbb{P}_m$)
grows exponentially, becoming prohibitively large as the order $m$ increases.
Although the exact convergence rate of $f_{2m}^{\star}$ to $f^{\star}$
remains unknown, we anticipate--based on the slow convergence observed in LP relaxations 
\cite{LasserreLP2002,LasserreLP2005} for scalar polynomial optimization derived from the 
Positivstellens\"atze of Handelman \eqref{eq::handelman} and Krivine–Stengle \eqref{eq::KS} 
--that the convergence $f_{2m}^{\star} \uparrow f^{\star}$ will also be slow.
Moreover, Huang and Nie \cite{HN2025} recently established the
finite convergence of ($\mathbb{P}_m$) under certain nondegeneracy and optimality
conditions. 
This suggests that SOS-based SDP relaxations ($\mathbb{P}_m$) are, in principle, 
superior to the relaxations $(\mathbb{H}_{2m})$.
However, it is important to observe that the size of PSD matrices 
in $(\mathbb{H}_{2m})$ depends only on the powers of $q$ and remains independent of 
the number $n$ of variables in ($\mathbb{P}$). In contrast, the sizes of two PSD 
matrices in ($\mathbb{P}_m$) approximately equal $n^m$ and $qn^m$.
Consequently,
for the problems ($\mathbb{P}$) with very large $n$ and relatively small $q$ 
that are significantly beyond the capability of the SDP relaxation ($\mathbb{P}_m$),
solving the alternatives $(\mathbb{H}_{2m})$ may still yield meaningful lower bounds of 
$f^{\star}$ in a reasonable time.

\begin{example}\label{ex::3}{\rm
   Consider the problem  ($\mathbb{P}$) where $n=\frac{q(q-1)}{2}$,
   $f(\bx)=-\sum_{i=1}^n x_i^2$ and 
   \[
   G(\bx)=[G_{ij}(\bx)]_{i,j\in[q]},\quad 
   G_{ij}(\bx)=\left\{
   \begin{array}{ll}
    1    &  \text{if}\ i=j,\\
    x_{\frac{(2q-i)(i-1)}{2}+j-i}    & \text{if}\ i<j,\\
    G_{ji}(\bx) & \text{if}\ i>j.\\
   \end{array}
   \right.
   \]
   For instance, when $n=3$, the problem becomes
   \[
   f^{\star}:=\min_{\bx\in\K} f(\bx)\quad\text{s.t. }\ \K=\left\{\bx\in\RR^3 \colon 
\left[
\begin{array}{ccc}
  1   & x_1 & x_2 \\
   x_1 & 1 & x_3\\
   x_2 &  x_3 & 1
\end{array}
\right]\succeq 0 \right\}. 
   \]
For any feasible point $\bx \in \K$, we have $x_i^2 \leq 1$ for all $i \in [n]$. 
Since the all-ones vector is feasible for ($\mathbb{P}$), it follows that $f^{\star} = -n$.
Now we apply the relaxations $(\mathbb{H}_{2m})$ and ($\mathbb{P}_m$) to ($\mathbb{P}$).
Like the relaxations $(\mathbb{H}_m)$, we implement ($\mathbb{P}_m$) using {\sf Yalmip}
and solve the resulting SDPs by {\tt Mosek}. 
To efficiently handle the representation of $f(\bx)-r$ required by ($\mathbb{P}_m$), 
we avoid the coefficient matching. 
Instead, we evaluate the polynomials at sufficiently many randomly generated generic 
points within the box $[-1, 1]^n$.

Note that we have $\deg(\Sigma)\le 2m-2$ in ($\mathbb{P}_m$).
Thus when $m=1$, the matrix $\Sigma$ in ($\mathbb{P}_m$) is constant. 
By comparing coefficients of $x_i^2$ ($i \in [n]$) in the representation of $f(\bx)-r$ 
in ($\mathbb{P}_m$), we observe that ($\mathbb{P}_m$) is infeasible for $m=1$. 
We therefore need to compute the second-order relaxation $(\mathbb{P}_2)$ in our experiments. 
However, $(\mathbb{P}_2)$ rapidly becomes numerically intractable as $q$ increases.
Thus, we introduce a variant $(\wt{\mathbb{P}}_1)$ of $(\mathbb{P}_1)$ 
where both $\sigma$ and $\Sigma$ in ($\mathbb{P}_m$) are constrained to degree 2, 
and denote by $\tilde{f}_{1}^{\sos}$ its optimal value. 
The size of the SDP problem $(\wt{\mathbb{P}}_1)$ lies between those of $(\mathbb{P}_1)$ 
and $(\mathbb{P}_2)$.  Numerically, we find that both of the relaxation $(\mathbb{H}_2)$ 
and $(\wt{\mathbb{P}}_1)$ are exact for ($\mathbb{P}$). 
Consequently, by increasing the number $q\in\N$ in ($\mathbb{P}$),
we compare the computational performance of two relaxations $(\wt{\mathbb{P}}_1)$
and $(\mathbb{H}_2)$. 
Note that as $q$ increases, the number $n$ of the variables becomes much larger than $q$. 
Numerical results are presented in Table \ref{tab::1}.

\qed
\begin{table}[htb]\caption{Numerical results for the problem in Example \ref{ex::3}.}
\label{tab::1}
\centering
\begin{tabular}{cccccc} 
\midrule[0.8pt]
$(q, n)$  & $(5, 10)$ & $(6, 15)$ & $(7, 21)$ & $(8, 28)$& $(9, 36)$\\
\midrule[0.4pt]
$\tilde{f}^{\sos}_1/\text{time}$&$-10.00$/2.12s &$-15.00/$14.4s &$-21.00/$120s
&$-28.00/$1h25m &$-/-$\\
$f^{\star}_1/\text{time}$&$-10.00/$0.91s &$-15.00/$1.28s &$-21.00/$1.67s
&$-28.00/$3.81s &$-36.00/$7.51s\\
\midrule[0.8pt]
\end{tabular}

\begin{tabular}{cccccc} 
\midrule[0.8pt]
$(q, n)$  & $(10, 45)$ & $(11, 55)$ & $(12, 66)$ & $(13, 78)$& $(14, 91)$\\
\midrule[0.4pt]
$\tilde{f}^{\sos}_1/\text{time}$&$-/-$ &$-/-$ &$-/-$
&$-/-$ &$-/-$\\
$f^{\star}_1/\text{time}$&$-45.00/$17.6s &$-55.00/$40.7s &$-66.00/$73.1s &$-78.00/$196s &$-91.00/$423s\\
\midrule[0.8pt]
\end{tabular}
\end{table}
   }
\end{example}

\section{Conclusions}\label{sec::con}
In this paper, we have established new non-SOS Positivstellensätze for polynomials positive on 
PMI-defined semialgebraic sets. By extending the classical theorems of Handelman and Krivine–Stengle 
to the matrix setting, we provided explicit representation forms for such polynomials without using 
sums of squares. 
Specifically, we showed that if is strictly positive on a compact PMI-defined set, 
then it can be written as a linear combination of Kronecker powers of the defining matrix 
(or its dilated form ) with PSD coefficient matrices. 
When the polynomial and the constraint matrix exhibit the correlative sparsity,
we proved that significantly simpler representations exist, 
drastically reducing the number and size of required PSD matrices.
By applying these algebraic certificates to polynomial optimization problems with PMI constraints, 
we constructed a new hierarchy of semidefinite programming relaxations. 
Although not exact in general, these relaxations remain computationally tractable even when 
the number of variables is large, provided the constaint matrix dimension is modest. 
For such problems, this offers a practical alternative to the traditional SOS-based hierarchies of 
SDP relaxations.

Several interesting questions remain for future work. 
It would be a primary challenge to derive degree bounds for the representations proposed 
in Theorems~\ref{th::main1} and~\ref{th::main2}.
Extending the present results to non-compact semialgebraic sets could broaden their applicability.
It would also be valuable to exploit additional sparsity structures 
(e.g., chordal sparsity \cite{zheng2019chordal}) in the polynomial matrix.
To tighten the proposed SDP relaxations in practice, one may incorporate bounded-degree SOS terms
\cite{LTY2017,WLT2018}, leading to a hybrid hierarchy.
Using computationally cheaper cones (e.g., diagonally dominant or scaled diagonally dominant matrices
\cite{AM2019}) 
instead of full PSD constraints in the proposed SDP relaxations to improve scalability while balancing 
solution quality.

\bibliographystyle{siamplain}
\bibliography{pcm}

\end{document}